\documentclass[11pt]{article}

\usepackage{amssymb,amsmath,amsfonts,amsthm}
\usepackage{latexsym}
\usepackage{graphics}
\usepackage{indentfirst}
\usepackage{hyperref}
\usepackage{comment}
\usepackage[shortlabels]{enumitem}
\usepackage[english]{babel}
\usepackage{tikz}
\usepackage{esdiff}

\usetikzlibrary{arrows.meta}

\newtheorem{innercustomthm}{Theorem}
\newenvironment{customthm}[1]
  {\renewcommand\theinnercustomthm{#1}\innercustomthm}
  {\endinnercustomthm}
\newtheorem{innercustomlem}{Lemma}
\newenvironment{customlem}[1]
  {\renewcommand\theinnercustomlem{#1}\innercustomlem}
  {\endinnercustomlem}

\newtheorem*{thm*}{Theorem}
\newtheorem{thm}{Theorem}
\newtheorem{lem}[thm]{Lemma}
\newtheorem{pro}[thm]{Proposition}
\newtheorem{obs}[thm]{Observation}
\newtheorem{cor}[thm]{Corollary}

\newtheorem{ques}[thm]{Question}
\newtheorem{defn}[thm]{Definition}

\newcommand{\HH}{\mathcal{H}}
\newcommand{\II}{\mathcal{I}}

\newcommand{\FF}{\mathcal{F}}
\newcommand{\DD}{\mathcal{D}}
\newcommand{\N}{\mathbb{N}}

\allowdisplaybreaks

\begin{document}

\title{The DP Color Function of Joins and Vertex-Gluings of Graphs}

\author{Jack Becker$^1$, Jade Hewitt$^1$, Hemanshu Kaul$^2$, Michael Maxfield$^1$, Jeffrey A. Mudrock$^1$, \\ David Spivey$^1$, Seth Thomason$^1$, and Tim Wagstrom$^1$ }

\footnotetext[1]{Department of Mathematics, College of Lake County, Grayslake, IL 60030.  E-mail:  {\tt {jmudrock@clcillinois.edu}}}

\footnotetext[2]{Department of Applied Mathematics, Illinois Institute of Technology, Chicago, IL 60616. E-mail: {\tt {kaul@iit.edu}}}

\maketitle

\begin{abstract}
DP-coloring (also called correspondence coloring) is a generalization of list coloring that has been widely studied in recent years after its introduction by Dvo\v{r}\'{a}k and Postle in 2015. As the analogue of the chromatic polynomial $P(G,m)$, the DP color function of a graph $G$, denoted $P_{DP}(G,m)$, counts the minimum number of DP-colorings over all possible $m$-fold covers. Chromatic polynomials for joins and vertex-gluings of graphs are well understood, but the effect of these graph operations on the DP color function is not known. In this paper we make progress on understanding the DP color function of the join of a graph with a complete graph and vertex-gluings of certain graphs. We also develop tools to study the DP color function under these graph operations, and we study the threshold (smallest $m$) beyond which the DP color function of a graph constructed with these operations equals its chromatic polynomial.
\medskip



\end{abstract}

\section{Introduction}\label{intro}

In this paper all graphs are nonempty, finite, simple graphs unless otherwise noted.  Generally speaking we follow West~\cite{W01} for terminology and notation.  The set of natural numbers is $\N = \{1,2,3, \ldots \}$.  For $m \in \N$, we write $[m]$ for the set $\{1, \ldots, m \}$.  Given a set $A$, $\mathcal{P}(A)$ is the power set of $A$.  We write \emph{AM-GM Inequality} for the inequality of arithmetic and geometric means.  If $G$ is a graph and $S, U \subseteq V(G)$, we use $G[S]$ for the subgraph of $G$ induced by $S$, and we use $E_G(S, U)$ for the set consisting of all the edges in $E(G)$ that have at least one endpoint in $S$ and at least one endpoint in $U$.  For $v \in V(G)$, we write $d_G(v)$ for the degree of vertex $v$ in the graph $G$ and $\Delta(G)$ for the maximum degree of $G$.  We write $N_G(v)$ (resp. $N_G[v]$) for the neighborhood (resp. closed neighborhood) of vertex $v$ in the graph $G$.  We use $\omega(G)$ to denote the clique number of the graph $G$.  When $C$ is a cycle on $n$ vertices ($n \geq 3$ since $C$ is simple), $V(C)= \{v_1, \ldots, v_n \}$, and $E(C) = \{\{v_1,v_2 \}, \{v_2, v_3 \}, \ldots, \{v_{n-1}, v_n \}, \{v_n, v_1 \} \}$, then we say the vertices of $C$ are written in \emph{cyclic order} when we write $v_1, \ldots, v_n$.  If $G$ and $H$ are vertex disjoint graphs, we write $G \vee H$ for the join of $G$ and $H$. When $G = K_1$, $G \vee H$ is the \emph{cone} of $H$, and the vertex in $V(G)$ is called the \emph{universal vertex} of $G \vee H$.

\subsection{List Coloring and DP-Coloring} \label{basic}

In the classical vertex coloring problem we wish to color the vertices of a graph $G$ with up to $m$ colors from $[m]$ so that adjacent vertices receive different colors, a so-called \emph{proper $m$-coloring}. The chromatic number of a graph $G$, denoted $\chi(G)$, is the smallest $m$ such that $G$ has a proper $m$-coloring.  List coloring is a variation on classical vertex coloring that was introduced independently by Vizing~\cite{V76} and Erd\H{o}s, Rubin, and Taylor~\cite{ET79} in the 1970s.  For list coloring, we associate a \emph{list assignment} $L$ with a graph $G$ such that each vertex $v \in V(G)$ is assigned a list of available colors $L(v)$ (we say $L$ is a list assignment for $G$).  We say $G$ is \emph{$L$-colorable} if there is a proper coloring $f$ of $G$ such that $f(v) \in L(v)$ for each $v \in V(G)$ (we refer to $f$ as a \emph{proper $L$-coloring} of $G$).  A list assignment $L$ is called a \emph{$k$-assignment} for $G$ if $|L(v)|=k$ for each $v \in V(G)$.  The \emph{list chromatic number} of a graph $G$, denoted $\chi_\ell(G)$, is the smallest $k$ such that $G$ is $L$-colorable whenever $L$ is a $k$-assignment for $G$.  We say $G$ is \emph{$k$-choosable} if $k \geq \chi_\ell(G)$.  Note $\chi(G) \leq \chi_\ell(G)$ since a $k$-assginment can assign the same colors to every vertex in $G$. This inequality may be strict since it is known that there are bipartite graphs with arbitrarily large list chromatic numbers (see~\cite{ET79}).

In 2015, Dvo\v{r}\'{a}k and Postle~\cite{DP15} introduced a generalization of list coloring called DP-coloring (they called it correspondence coloring) in order to prove that every planar graph without cycles of lengths 4 to 8 is 3-choosable. DP-coloring has been extensively studied over the past 5 years (see e.g.,~\cite{B16,B17, BK17, BK18, HK21, KM19, KM20, KM21, KO182, M18, MT20}).  Intuitively, DP-coloring is a variation on list coloring where each vertex in the graph still gets a list of colors, but identification of which colors are different can change from edge to edge.  Following~\cite{BK17}, we now give the formal definition. Suppose $G$ is a graph.  A \emph{cover} of $G$ is a pair $\mathcal{H} = (L,H)$ consisting of a graph $H$ and a function $L: V(G) \rightarrow \mathcal{P}(V(H))$ satisfying the following four requirements:

\vspace{5mm}

\noindent(1) the set $\{L(u) : u \in V(G) \}$ is a partition of $V(H)$ of size $|V(G)|$; \\
(2) for every $u \in V(G)$, the graph $H[L(u)]$ is complete; \\
(3) if $E_H(L(u),L(v))$ is nonempty, then $u=v$ or $uv \in E(G)$; \\
(4) if $uv \in E(G)$, then $E_H(L(u),L(v))$ is a matching (the matching may be empty).

\vspace{5mm}

Suppose $\mathcal{H} = (L,H)$ is a cover of $G$.  We refer to the edges of $H$ connecting distinct parts of the partition $\{L(v) : v \in V(G) \}$ as \emph{cross-edges}.  An \emph{$\mathcal{H}$-coloring} of $G$ is an independent set in $H$ of size $|V(G)|$.  It is immediately clear that an independent set $I \subseteq V(H)$ is an $\mathcal{H}$-coloring of $G$ if and only if $|I \cap L(u)|=1$ for each $u \in V(G)$.  We say $\mathcal{H}$ is \emph{$m$-fold} if $|L(u)|=m$ for each $u \in V(G)$.  An $m$-fold cover $\mathcal{H}$ is a \emph{full cover} if for each $uv \in E(G)$, the matching $E_{H}(L(u),L(v))$ is perfect.  The \emph{DP-chromatic number} of $G$, $\chi_{DP}(G)$, is the smallest $m \in \N$ such that $G$ has an $\mathcal{H}$-coloring whenever $\mathcal{H}$ is an $m$-fold cover of $G$.

Suppose $\mathcal{H} = (L,H)$ is an $m$-fold cover of $G$.  We say that $\mathcal{H}$ has a \emph{canonical labeling} if it is possible to name the vertices of $H$ so that $L(u) = \{ (u,j) : j \in [m] \}$ and $(u,j)(v,j) \in E(H)$ for each $j \in [m]$ whenever $uv \in E(G)$.~\footnote{When $\mathcal{H}=(L,H)$ has a canonical labeling, we will always refer to the vertices of $H$ using this naming scheme.}  Clearly, when $\mathcal{H}$ has a canonical labeling, $m \geq \chi(G)$, $\mathcal{I}$ is the set of $\mathcal{H}$-colorings of $G$, and $\mathcal{C}$ is the set of proper $m$-colorings of $G$, the function $f: \mathcal{C} \rightarrow \mathcal{I}$ given by $f(c) = \{ (v, c(v)) : v \in V(G) \}$ is a bijection.  Also, given an $m$-assignment $L$ for a graph $G$, it is easy to construct an $m$-fold cover $\mathcal{H}'$ of $G$ such that $G$ has an $\mathcal{H}'$-coloring if and only if $G$ has a proper $L$-coloring (see~\cite{BK17}).  It follows that $\chi(G) \leq \chi_\ell(G) \leq \chi_{DP}(G)$.  The second inequality may be strict, e.g., it is easy to prove that $\chi_{DP}(C_n) = 3$ whenever $n \geq 3$, but the list chromatic number of any even cycle is 2 (see~\cite{BK17} and~\cite{ET79}).  In some instances DP-coloring behaves similar to list coloring, but there are some interesting differences.  Much of the research on DP-coloring explores its similarities and differences with list coloring (see e.g.,~\cite{B16, B17, BK17}).   
 
\subsection{Counting Proper Colorings and List Colorings}

In 1912 Birkhoff introduced the notion of the chromatic polynomial with the hope of using it to make progress on the four color problem.  For $m \in \N$, the \emph{chromatic polynomial} of a graph $G$, $P(G,m)$, is the number of proper $m$-colorings of $G$. It is well-known that $P(G,m)$ is a polynomial in $m$ of degree $|V(G)|$ (see~\cite{B12}).   For example, $P(K_n,m) = \prod_{i=0}^{n-1} (m-i)$, $P(C_n,m) = (m-1)^n + (-1)^n (m-1)$, and $P(T,m) = m(m-1)^{n-1}$ whenever $T$ is a tree on $n$ vertices (see~\cite{W01}). 

We now mention two chromatic polynomial formulas that will be important in this paper. First, if $G$ is an arbitrary graph and $n \in \N$, then $P(G \vee K_n, m) = P(K_n, m)P(G, m-n)$ whenever $m \geq n+1$ (see~\cite{B94}).  Second, suppose that $n \geq 2$, $G_1, \ldots, G_n$ are vertex disjoint graphs, and $1 \le p \le \min_i\{\omega(G_i)\}$.  Choose a copy of $K_p$ contained in each $G_i$ and form a new graph $G$, called the \emph{$K_p$-gluing of $G_1, \ldots, G_n$}, from the union of $G_1, \ldots, G_n$ by arbitrarily identifying the chosen copies of $K_p$; that is, if $\{u_{i,1}, \ldots, u_{i,p} \}$ is the vertex set of the chosen copy of $K_p$ in $G_i$ for each $i \in [n]$, then identify the vertices $u_{1,j}, \ldots, u_{l,j}$ as a single vertex $u_j$ for each $j \in [p]$.~\footnote{This is equivalent to the clique-sum of  $G_1, \ldots, G_n$ where no edges are removed after the identification of the cliques.} When $p=1$ this is called \emph{vertex-gluing} and when $p=2$ this is called \emph{edge-gluing} (see~\cite{DKT05}). Let $\bigoplus_{i=1}^{n}(G_i,p)$ denote the family of all $K_p$-gluings of $G_1, \ldots, G_n$. It is well-known that for any $G\in \bigoplus_{i=1}^{n}(G_i,p)$, $P(G,m) = \prod_{i=1}^n P(G_i,m)/ \left(\left( \prod_{i=0}^{p-1} (m-i) \right)^{n-1} \right)$ whenever $m \geq p$ (see~\cite{B94, DKT05}).

The notion of chromatic polynomial was extended to list coloring in the early 1990s~\cite{AS90}. If $L$ is a list assignment for $G$, we use $P(G,L)$ to denote the number of proper $L$-colorings of $G$. The \emph{list color function} $P_\ell(G,m)$ is the minimum value of $P(G,L)$ where the minimum is taken over all possible $m$-assignments $L$ for $G$.  Clearly, $P_\ell(G,m) \leq P(G,m)$ for each $m \in \N$.  In general, the list color function can differ significantly from the chromatic polynomial for small values of $m$.  However, for large values of $m$, Wang, Qian, and Yan~\cite{WQ17} (improving upon results in~\cite{D92} and~\cite{T09}) showed the following in 2017.

\begin{thm} [\cite{WQ17}] \label{thm: WQ17}
If $G$ is a connected graph with $l$ edges, then $P_{\ell}(G,m)=P(G,m)$ whenever $m > \frac{l-1}{\ln(1+ \sqrt{2})}$.
\end{thm}

It is also known that $P_{\ell}(G,m)=P(G,m)$ for all $m \in \N$ when $G$ is a cycle or chordal (see~\cite{KN16} and~\cite{AS90}).  Moreover, if $P_{\ell}(G,m)=P(G,m)$ for all $m \in \N$, then $P_{\ell}(K_n \vee G,m)=P(K_n \vee G,m)$ for each $n, m \in \N$ (see~\cite{KM18}). See~\cite{T09} for a survey of known results and open questions on the list color function.

\subsection{The DP Color Function and Motivating Questions}

Two of the current authors (Kaul and Mudrock in~\cite{KM19}) introduced a DP-coloring analogue of the chromatic polynomial to gain a better understanding of DP-coloring and use it as a tool for making progress on some open questions related to the list color function.  Suppose $\mathcal{H} = (L,H)$ is a cover of graph $G$.  Let $P_{DP}(G, \mathcal{H})$ be the number of $\mathcal{H}$-colorings of $G$.  Then, the \emph{DP color function} of $G$, $P_{DP}(G,m)$, is the minimum value of $P_{DP}(G, \mathcal{H})$ where the minimum is taken over all possible $m$-fold covers $\mathcal{H}$ of $G$.~\footnote{We take $\N$ to be the domain of the DP color function of any graph.} It is easy to show that for any graph $G$ and $m \in \N$, $P_{DP}(G, m) \leq P_\ell(G,m) \leq P(G,m)$.~\footnote{To prove this, recall that for any $m$-assignment $L$ for $G$, an $m$-fold cover $\mathcal{H}'$ of $G$ such that $G$ has an $\mathcal{H}'$-coloring if and only if $G$ has a proper $L$-coloring is constructed in~\cite{BK17}. It is easy to see from the construction in~\cite{BK17} that there is a bijection between the proper $L$-colorings of $G$ and the $\mathcal{H}'$-colorings of $G$.} Note that if $G$ is a disconnected graph with components: $H_1, H_2, \ldots, H_t$, then $P_{DP}(G, m) = \prod_{i=1}^t P_{DP}(H_i,m)$.  So, we will only consider connected graphs from this point forward unless otherwise noted.

As with list coloring and DP-coloring, the list color function and DP color function of certain graphs behave similarly.  However, for some graphs there are surprising differences.  For example, similar to the list color function,  $P_{DP}(G,m) = P(G,m)$ for every $m \in \N$  whenever $G$ is chordal or an odd cycle~\cite{KM19}.  On the other hand, we have the following result.

\begin{thm} [\cite{KM19}] \label{thm: evengirth}
If $G$ is a graph with girth that is even, then there is an $N \in \N$ such that $P_{DP}(G,m) < P(G,m)$ whenever $m \geq N$.  Furthermore, for any integer $g \geq 3$ there exists a graph $M$ with girth $g$ and an $N \in \N$ such that $P_{DP}(M,m) < P(M,m)$ whenever $m \geq N$. 
\end{thm}

This result is particularly interesting since Theorem~\ref{thm: WQ17} implies that the list color function of any graph eventually equals its chromatic polynomial.  It is also interesting to note that the key to showing there is a graph $G$ with girth $g$ that is odd such that $P_{DP}(G,m)$ does not eventually equal its chromatic polynomial, is to take an edge-gluing of a graph with girth that is odd and a sufficiently large even cycle. 

In contrast, the following recent result tells us that the DP color function of the cone of a graph behaves like the list color function in that it eventually equals its chromatic polynomial.

\begin{thm} [\cite{MT20}] \label{thm: seth}
For any graph $G$, $P(G,m)-P_{DP}(G,m) = O(m^{n-3})$ as $m \rightarrow \infty$.  Moreover, there exists an $N \in \N$ such that $P_{DP}(K_1 \vee G, m) = P(K_1 \vee G, m)$ whenever $m \geq N$. 
\end{thm}

However, unlike the bound provided in Theorem~\ref{thm: WQ17}, given a graph $G$, little is known about the $N$ in Theorem~\ref{thm: seth} (see~\cite{MT20}).  With this in mind, for any graph $G$, we define the \emph{DP color function threshold of G}, $\tau_{DP}(G)$, to be the smallest $N \geq \chi(G)$ such that $P_{DP}(G,m) = P(G,m)$ whenever $m \geq N$.  If $P(G,m) - P_{DP}(G,m) > 0$ for infinitely many $m$, we let $\tau_{DP}(G) = \infty$.  We now state an important open question related to the DP color function threshold~\footnote{The list color function analogue of Question~\ref{ques: mono} is also open~\cite{KN16}.}.

\begin{ques} [\cite{KM19}] \label{ques: mono}
If $P_{DP}(G,m_0) = P(G,m_0)$ for some $m_0 \geq \chi(G)$, does it follow that $\tau_{DP}(G) \leq m_0$?
\end{ques}

An equivalent way to ask Question~\ref{ques: mono} is: If $P_{DP}(G,m_0) = P(G,m_0)$ for some $m_0 \geq \chi(G)$, then does it follow that $P_{DP}(G,m_0 + 1) = P(G,m_0 + 1)$?  

Our focus in this paper is on questions that arise naturally from the ideas discussed above.  Specifically, with Theorem~\ref{thm: seth} and Question~\ref{ques: mono} in mind, we wish to pursue the following question.

\begin{ques} \label{ques: join}
Given an arbitrary graph $G$ and $p \in \N$, what is the value of $\tau_{DP}(K_p \vee G)$?
\end{ques}

Note that $\tau_{DP}(K_p \vee G)$ must be finite for any $p$ and $G$ by Theorem~\ref{thm: seth}. Since the proof of Theorem~\ref{thm: evengirth} tells us that an edge-gluing of certain graphs can have DP color function that does not eventually equal its chromatic polynomial, it is natural to ask whether the formula for the chromatic polynomial of $K_p$-gluings of graphs has a DP color function analogue.

\begin{ques} \label{ques: classify}
Given $p \ge 1$, $n \geq 2$, vertex disjoint graphs $G_1, \ldots, G_n$, and $G \in \bigoplus_{i=1}^{n}(G_i,p)$, is it the case that 
$$P_{DP}(G,m) \leq \frac{\prod_{i=1}^n P_{DP}(G_i,m)}{\left( \prod_{i=0}^{p-1} (m-i) \right)^{n-1}}$$ 
for all $m \geq p$? Moreover, for which vertex disjoint graphs $G_1, \ldots, G_n$, and which $G \in \bigoplus_{i=1}^{n}(G_i,p)$ is this bound tight?
\end{ques} 

We will see below that  the inequality in Question~\ref{ques: classify} need not be an equality in all situations: when $p=1$, $n=2$, $G_1 = G_2 = K_1 \vee C_4$, and $G$ is formed by gluing the universal vertices of $G_1$ and $G_2$ respectively, then $P_{DP}(G,4) < (P_{DP}(G_1,4)P_{DP}(G_2,4))/4$.  The focus of the last section of this paper is on making progress on Question~\ref{ques: classify}. 

\subsection{Outline of Results}

In Section~\ref{join} we begin by exploring the relationship between the DP color function of $K_p \vee G$ and $K_{p+1} \vee G$. We prove the following.

\begin{thm} \label{thm: monoinp}
Suppose $P_{DP}(K_p \vee G, m) = P(K_p \vee G, m)$ whenever $m \geq N$.  Then, $P_{DP}(K_{p+1} \vee G, s) = P(K_{p+1} \vee G, s)$ whenever $s \geq N+1$.  Consequently, $\tau_{DP}(K_{p+1} \lor G) \leq \tau_{DP}(K_p \lor G) + 1$. 
\end{thm}

Then, we develop techniques to completely determine the DP color functions of wheels (i.e., the cone of a cycle).

\begin{thm} \label{thm: wheel} 
For any $k,m \in \N$, $P_{DP}(K_1\lor C_{2k+1},m)=P(K_1\lor C_{2k+1},m)$, and 
\[
  P_{DP}(K_1 \lor C_{2k+2}, m) =
  \begin{cases}
                                   0 & \text{if $m \in [2]$} \\
                                   3 & \text{if $m=3$} \\ 
  P(K_1 \lor C_{2k+2}, m) & \text{if $m \geq 4$}.
  \end{cases}
\]
\end{thm}

Notice that Theorem~\ref{thm: wheel} tells us that the answer to Question~\ref{ques: mono} is yes when we restrict our attention to wheels.  We end Section~\ref{join} by answering Question~\ref{ques: join} in the case of the join of a complete graph and cycle.

\begin{thm} \label{thm: threshold}
For any $p \in \N$ and $n \geq 3$, $\tau_{DP}(K_p \vee C_n) = 3+p$.
\end{thm}

Theorem~\ref{thm: threshold} also demonstrates that the bound on $\tau_{DP}(K_{p+1} \lor G)$ in Theorem~\ref{thm: monoinp} is tight.
  
In Section~\ref{tools}, we build a toolbox for studying $K_p$-gluings of graphs. Given vertex disjoint graphs $G_1, \ldots, G_n$, we define amalgamated cover, a natural analogue of ``gluing'' $m$-fold covers of each $G_i$ together so that we get an $m$-fold cover for $G\in \bigoplus_{i=1}^{n}(G_i,1)$. We define separated covers, a natural analogue of ``splitting'' an $m$-fold cover of $G\in \bigoplus_{i=1}^{n}(G_i,p)$ into separate $m$-fold covers for each $G_i$. These notions help us show that the inequality in Question~\ref{ques: classify} holds in the case that $p=1$.  

\begin{thm}\label{thm: upperbound}
Suppose that $G_1,\ldots,G_n$ are vertex disjoint graphs for some $n \geq 2$ 
and $G \in \bigoplus_{i=1}^{n}(G_i,1)$. Then $$P_{DP}(G,m) \leq \frac{\prod_{i=1}^{n} P_{DP}(G_i,m)}{m^{n-1}}.$$
\end{thm}

Next, we prove the following general result which is helpful for establishing lower bounds on the DP color function of $K_p$-gluings of graphs.

\begin{lem}\label{lem: lowerGen}
Suppose that $G_1,\ldots,G_n$ are vertex disjoint graphs where $n \geq 2$ and $G \in \bigoplus_{i=1}^n(G_i,p)$ where for each $i \in [n]$, $\{u_{i,1},\ldots,u_{i,p}\}$ is a clique in $G_i$ and $G$ is obtained by identifying $u_{1,q},\ldots,u_{n,q}$ as the same vertex $u_q$ for each $q \in [p]$.  Suppose that for each $i \in [n]$, given any $m$-fold cover $\DD_i = (K_i,D_i)$ of $G_i$, $A$ is contained in at least $k_i$ $\DD_i$-colorings of $G_i$ whenever $A \subseteq \bigcup_{q=1}^{p} K_i(u_{i,q})$, $|A \cap K_i(u_{i,q})| = 1$ for each $q \in [p]$, and $A$ is an independent set in $D_i$. Then $$P_{DP}(G,m) \geq \left(\prod_{i=0}^{p-1} (m-i)\right) \left(\prod_{i=1}^{n} k_i\right).$$
\end{lem}

In Section~\ref{CycleChordal}, we use Theorem~\ref{thm: upperbound} and Lemma~\ref{lem: lowerGen} to show the following which makes progress on Question~\ref{ques: classify}.

\begin{thm}\label{thm: ChordCycle}
Suppose that $G_1,\ldots,G_n$ are vertex disjoint graphs, and each $G_i$ is either a chordal graph or a cycle. For any $G \in \bigoplus_{i=1}^{n}(G_i,1)$ and $m \in \N$, $$P_{DP}(G,m) = \frac{\prod_{i=1}^{n} P_{DP}(G_i,m)}{m^{n-1}}.$$
\end{thm}

We end with Section~\ref{ConeCycles} by studying vertex-gluings of wheels (equivalently, cones of the disjoint union of cycles). In doing so, we determine the DP color function threshold of all such graphs which makes further progress on Questions~\ref{ques: mono} and~\ref{ques: classify}.

\section{Joins of Graphs} \label{join}

We begin by recalling a basic result and establishing some terminology and notation that will be useful for the remainder of this paper.  The following basic result will be used frequently.

\begin{pro} [\cite{KM19}] \label{pro: tree}
Suppose $T$ is a tree and $\mathcal{H} = (L,H)$ is a full $m$-fold cover of $T$.  Then, $\mathcal{H}$ has a canonical labeling.
\end{pro}
  
Now, suppose $M = K_1 \lor G$ for some graph $G$.  Whenever $\mathcal{H} = (L, H)$ is a full $m$-fold cover of $M$, we use the following conventions from this point forward unless otherwise noted.  We assume that $w$ is the vertex corresponding to the copy of $K_1$ used to form $M$, and the vertices of $H$ are named so that $L(v) = \{(v,j): j \in [m]\}$ for each $v \in V(M)$ and $E_H(L(w),L(u)) = \{(w,j)(u,j): j \in [m]\}$ whenever $u \in V(G)$ (this is permissible by Proposition~\ref{pro: tree}). When all of these conventions are followed, we let the \emph{cone reduction of $\HH$ by $(w,j)$} be $\HH^{(j)} = (L^{(j)},H^{(j)})$ where $L^{(j)}(v) = L(v) - \{(v,j') : (v,j') \in N_H((w,j))\}$ and $H^{(j)} = H - N_H[(w,j)]$. Note $\mathcal{H}^{(j)}$ is an $(m-1)$-fold cover of $G$, and $\mathcal{H}^{(j)}$ is not necessarily full.  We say that $(w,t) \in L(w)$ is a \textit{level vertex} if $H^{(t)}$ contains precisely $|E(G)|(m-1)$ cross-edges (i.e., $\mathcal{H}^{(t)}$ is full).

The following lemma captures the essence of the proof of Theorem~\ref{thm: monoinp}.

\begin{lem} \label{lem: joinEquality}
Suppose that for some graph $G$ and $m \in \mathbb{N}$, $P_{DP}(G,m)=P(G,m)$. Then, $P_{DP}(K_1 \lor G,m+1)=P(K_1 \lor G,m+1)$.
\end{lem}

\begin{proof}
Let $M = K_1 \lor G$.  Clearly, $P(M,m+1)=(m+1)P(G,m)$.  We assume that $\mathcal{H} = (L, H)$ is an $(m+1)$-fold cover of $M$ such that $P_{DP}(M, \mathcal{H}) = P_{DP}(M,m+1)$, and we can assume that $\mathcal{H}$ is full since adding edges to a graph can not increase the number of independent sets of a given size.  To prove the desired we must show  $P_{DP}(M, \mathcal{H}) \geq P(M, m+1)$.

Let $\mathcal{I}_j$ be the set of $\mathcal{H}$-colorings containing $(w,j)$. It is clear that $P_{DP}(M,\mathcal{H}) = \sum_{j=1}^{m+1} |\mathcal{I}_j|$. Let $\HH^{(j)}$ be the cone reduction of $\HH$ by $(w,j)$.
Note that $|\mathcal{I}_j| = P_{DP}(G, \mathcal{H}^{(j)})$.  Since $\mathcal{H}^{(j)}$ is an $m$-fold cover of $G$, we know that $P_{DP}(G,m) \leq P_{DP}(G,\mathcal{H}^{(j)})$.  So,
\begin{align*}
    P_{DP}(M,m+1) = \sum_{j=1}^{m+1} |\mathcal{I}^{(j)}| = \sum_{j=1}^{m+1} P_{DP}(G, \mathcal{H}^{(j)}) &\geq \sum_{j=1}^{m+1} P_{DP}(G, m) \\
		&= (m+1)P_{DP}(G, m) \\
		&= (m+1)P(G,m) = P(M,m+1)
\end{align*}
which completes the proof.
\end{proof}

We can now finish the proof of Theorem~\ref{thm: monoinp}.  

\begin{proof}
It is clear that $K_{p+1} \lor G = K_1 \lor (K_p \lor G)$. Suppose $M = K_p \lor G$, and assume $m$ is a fixed integer satisfying $m \geq N$. We have that $P_{DP}(M,m) = P(M,m)$. Then, Lemma~\ref{lem: joinEquality} implies that $P_{DP}(K_1 \lor M, m+1) = P(K_1 \lor M, m+1)$. The desired result immediately follows.
\end{proof}

The next lemma will take care of the situation where the cycle is odd in Theorems~\ref{thm: wheel} and~\ref{thm: threshold}.

\begin{lem}
For any $p, k \in \N$, $P_{DP}(K_p\lor C_{2k+1},m)=P(K_p\lor C_{2k+1},m)$ whenever $m\in\mathbb{N}$.  Consequently, $\tau_{DP}(K_p \vee C_{2k+1}) = \chi(K_p \vee C_{2k+1}) = 3+p$.
\end{lem}

\begin{proof}
Let $G=C_{2k+1}$ for some $k\in \mathbb{N}$.  We know $P_{DP}(G,m) = P(G,m)$ for any $m \in \mathbb{N}$ (see~\cite{KM19}).  Now, we will prove the desired by induction on $p$. When $p=1$, Theorem~\ref{thm: monoinp} implies $P_{DP}(K_1 \lor G,n) = P(K_1 \lor G,n)$ for any $n \geq 2$.  The desired result holds since $0 \leq P_{DP}(K_1 \lor G,1) \leq P(K_1 \lor G,1)=0$.  This completes the basis step.

Now, suppose $p \geq 2$ and the desired result holds for all positive integers less than $p$.  This means $P_{DP}(K_{p-1} \vee G,m) = P(K_{p-1} \vee G,m)$ for any $m \in \mathbb{N}$.  Lemma~\ref{lem: joinEquality} implies $P_{DP}(K_p \vee G,n)=P_{DP}(K_1 \lor (K_{p-1} \vee G),n) = P(K_1 \lor (K_{p-1} \vee G),n) = P(K_p \vee G,n)$ for any $n \in \mathbb{N}-\{1\}$.  The desired result holds since $0 \leq P_{DP}(K_p \lor G,1) \leq P(K_p \lor G,1)=0$.  This completes the induction step.   
\end{proof}

We now turn our attention to cones of even cycles.  In this Section, whenever $G$ is a copy of $C_{2k+2}$, we assume that the vertices of $G$ in cyclic order are $v_1, v_2, \ldots, v_{2k+2}$.

\begin{lem} \label{lem: sethFormula}
Suppose $M = K_1 \lor G$ where $G = C_{2k+2}$ and $k \in \mathbb{N}$.  Suppose $\mathcal{H} = (L, H)$ is an arbitrary $m$-fold cover of $M$ where $m \geq 2$.  Let $s$ be the number of level vertices in $L(w)$.
Also let $p = P(P_{2k+2},m-1) = (m-1)(m-2)^{2k+1}$,\\
$p_1 = P(C_{2k+1}, m-1)/(m-1) = ((m-2)^{2k+1} - (m-2))/(m-1)$, and\\
$p_2 = P(C_{2k+2}, m-1)/((m-1)(m-2)) = ((m-2)^{2k+1} + 1)/(m-1)$.\\
For all $(w,j) \in L(w)$, the following statements hold.\\
(i) If $(w,j)$ is a level vertex, then $(w,j)$ is in at least $(p - (s-1)p_1 - (m-s) p_2)$ $\mathcal{H}$-colorings of $M$.\\
(ii) If $(w,j)$ is not a level vertex, then $(w,j)$ is in at least $(p - sp_1 - (m-2-s)p_2)$ $\mathcal{H}$-colorings of $M$. \\
Consequently, $P_{DP}(M,m)$ is at least
\begin{align*}
&\min_{s \in \{0\} \cup ([m]-\{m-1\})} \left(s(p - (s-1)p_1 - (m-s)p_2) + (m-s)(p - sp_1 - (m-2-s)p_2)\right).
\end{align*}
\end{lem}

\begin{proof}
Let $\HH^{(j)}$ be the cone reduction of $\HH$ by $(w,j)$. Clearly, the number of $\mathcal{H}$-colorings of $M$ containing $(w,j)$ is equal to the number of $\mathcal{H}^{(j)}$-colorings of $G$ (Note that when $(w,j)$ is not a level vertex, we will assume without loss of generality that $E_{H^{(j)}}(L(v_1),L(v_{2k+2}))$ is not a perfect matching).  Let $H_l = H^{(j)}[\{(v_i,l): i \in [2k+2]\}]$ for each $l \in [m] - \{j\}$.  Let $c_l = 1$ if $H_l = C_{2k+2}$ and $0$ otherwise.  Now, let $S = \sum_{l \in [m] - \{j\}} c_l$.  Notice $S = s-1$ when $(w,j)$ is a level vertex, and $S = s$ when $(w,j)$ is not a level vertex.
We will now rename the vertices in $V(H^{(j)})$ according to the following procedure.

Suppose $J$ is the graph obtained from modifying $H^{(j)}$ as follows. Delete the edges in $E_{H^{(j)}}(L(v_1),L(v_{2k+2}))$.  For each $i \in [2k+1]$, add edges (if needed) so that $E_{H^{(j)}}(L(v_i),L(v_{i+1}))$ is a perfect matching. This completes the construction of $J$.  Notice $V(J) = V(H^{(j)})$.  Since $G - \{v_1v_{2k+2}\}$ is a path, $J$ can be decomposed into $m-1$ paths $D_i$ where $D_i$ is the path containing the vertex $(v_1, i)$ for each $i \in [m]-\{j\}$.  Furthermore, for each $l \in [m]-\{j\}$ and $i \in [2k+2]$, $|V(D_l) \cap L^{(j)}(v_i)| = 1$.  Also, let $(v_i, a_{l,i})$ be the element in $V(D_l) \cap L^{(j)}(v_i)$. Finally, rename each element of $V(H^{(j)})$ so that each $(v_i, a_{l,i})$ is renamed $(v_i^*, l)$.
Let $A$ be the set of vertices of $H^{(j)}$ before the renaming, and $B$ be the set of vertices of $H^{(j)}$ after the renaming.
Let $R: A \rightarrow B$ be the function such that $R((v_i,l))$ is the name of $(v_i, l)$ under the new naming scheme.

Now, the vertices of $H^{(j)}$ are named such that $L^{(j)}(v_i) = \{(v_i^*, l): l \in [m]-\{j\}\}$.  Let $H_l^* = H^{(j)}[\{(v_i^*,l): i \in [2k+2]\}]$ for each $l \in [m] - \{j\}$.  Let $b_l = 1$ if $H_l^* = C_{2k+2}$ and $0$ otherwise.  Now, let $T = \sum_{l \in [m] - \{j\}} b_l$.  Notice that if $c_l=1$ for some $l \in [m] - \{j\}$, then $R((v_i, l)) = (v_i^*, l)$ for each $i \in [2k+2]$ which implies that $b_l=1$.  Consequently, $T \geq S$.

We will now prove Statement~(i).  Let $H' = H^{(j)} - E_{H^{(j)}}(L(v_1),L(v_{2k+2}))$. Let $\mathcal{H'} = (L^{(j)}, H')$.  Notice $\mathcal{H'}$ is an $(m-1)$-fold cover of $G - \{v_1v_{2k+2}\}$.  Furthermore, based on the renaming scheme, finding an $H'$-coloring of $G-\{v_1v_{2k+2}\}$ is equivalent to finding a proper $(m-1)$-coloring of $G-\{v_1v_{2k+2}\}$. Consequently, there are $P(P_{2k+2}, m-1) = p$,  $\mathcal{H'}$-colorings of $G - \{v_1v_{2k+2}\}$. It is clear that an $\mathcal{H'}$-coloring $I$ of $G$ is not a $\mathcal{H}$-coloring of $G$ if and only if $I$ contains a vertex in $L^{(j)}(v_1)$ and a vertex in $L^{(j)}(v_{2k+2})$ that are adjacent in $H^{(j)}$.  Notice that $(v_1^*,i)$ and $(v_{2k+2}^*, i)$ are in precisely $P(C_{2k+1}, m-1)/(m-1) = p_1$, $\mathcal{H'}$-colorings of $G - \{v_1v_{2k+2}\}$.  Similarly, notice that $(v_1^*,i)$ and $(v_{2k+2}^*, j)$ where $i \neq j$ are in precisely $P(C_{2k+2}, m-1) / ((m-1)(m-2)) = p_2$, $\mathcal{H'}$-colorings of $G - \{v_1v_{2k+2}\}$.  Clearly, $p_2 > p_1$.  Note that $E_{H^{(j)}}(L(v_1),L(v_{2k+2}))$ has precisely $T$ edges that connect two vertices with the same second coordinate and $m-1-T$ edges that connect two vertices with different second coordinates.  The number of $\mathcal{H}^{(j)}$-colorings of $G$ is $p - Tp_1 - (m - 1 - T)p_2$.
Since $T \geq S = s-1$ and $p_2 \geq p_1$, it follows that
\begin{align*}
    p - Tp_1 - (m - 1- T)p_2 = p + T(p_2 - p_1) - (m-1)p_2 &\geq p + (s-1)(p_2 - p_1) - (m-1)p_2\\
    &= p - (s-1)p_1 - (m-s)p_2.
\end{align*}
Now we turn our attention to Statement~(ii).  Following the same idea as the proof of Statement~(i), note that $E_{H^{(j)}}(L(v_1),L(v_{2k+2}))$ has precisely $T$ edges that connect two vertices with the same second coordinate.  Since $(w,j)$ is a non-level vertex and we assumed that $E_{H^{(j)}}(L(v_1),L(v_{2k+2}))$ is not a perfect matching, there are at most $m-2-T$ edges in $E_{H^{(j)}}(L(v_1),L(v_{2k+2}))$ that connect two vertices with different second coordinates.  The number of $\mathcal{H}^{(j)}$-colorings of $G$ is at least $p - Tp_1 - (m - 2 - T)p_2$.  Since $T \geq S = s$ and $p_2 \geq p_1$, it follows that
\begin{align*}
    p - Tp_1 - (m - 2 - T)p_2 = p + T(p_2 - p_1) - (m - 2)p_2 &\geq p + s(p_2 - p_1) - (m-2)p_2\\
    &= p - sp_1 - (m - 2 - s)p_2.
\end{align*}

Finally, the fact that  $P_{DP}(M,m)$ is at least
\begin{align*}
&\min_{s \in \{0\} \cup ([m]-\{m-1\})} \left(s(p - (s-1)p_1 - (m-s)p_2) + (m-s)(p - sp_1 - (m-2-s)p_2)\right)
\end{align*}
immediately follows from Statements~(i) and~(ii) and the fact the number of level vertices of any $m$-fold cover of $M$ must be in the set $\{0\} \cup ([m]- \{m-1\})$ (see~\cite{KM19}).
\end{proof}

We are now ready to complete the proof of Theorem~\ref{thm: wheel} by proving it for cones of even cycles.

\begin{proof}
Suppose $m \geq 4$, and $M = K_1 \lor G$ where $G = C_{2k+2}$. To apply Lemma~\ref{lem: sethFormula}, we need to compute:
\begin{align*}
    &s(p - (s-1)p_1 - (m-s)p_2) + (m-s)(p - sp_1 - (m-2-s)p_2)\\
    &= s(p_1(1 - m) + p_2(m - 2)) + (mp - m^2p_2 + 2mp_2)\\
    &= s\left((m-2)-(m-2)^{2k+1} + \frac{(m-2)^{2k+2} + (m-2)}{m-1}\right) + mp - m^2p_2 + 2mp_2.
\end{align*}
Notice that since $m \geq 4$ and $k \in \mathbb{N}$, $m - (m-2)^{2k} \leq 0$.
So, 
\[
    (m-2)-(m-2)^{2k+1} + \frac{(m-2)^{2k+2} + (m-2)}{m-1} =
    \frac{m(m-2) - (m-2)^{2k+1}}{m-1} \leq 0.
\]
Thus,
\begin{align*}
    &\min_{s \in \{0\} \cap ([m]-\{m-1\})} \left(s(p - (s-1)p_1 - (m-s)p_2) + (m-s)(p - sp_1 - (m-2-s)p_2)\right)\\
    &= m(p - (m-1)p_1) = m((m-2)^{2k+2} + (m-2)) = P(M,m).
\end{align*}
It immediately follows by Lemma~\ref{lem: sethFormula} that $P_{DP}(M,m) = P(M,m)$ whenever $m \geq 4$.

Now, we will show that $P_{DP}(M,3) = 3$.  For each $v \in V(M)$, let $L(v) = \{(v,j): j \in [3]\}$.  Let $H$ be the graph with vertex set $\bigcup_{v \in V(M)} L(v)$.  We will now describe the construction of edges in $H$.  Create edges so that $H[L(v)]$ is a complete graph for each $v \in V(M)$. Whenever $uv \in E(M) - \{v_1v_{2k+2}\}$, create an edge between $(u,j)$ and $(v,j)$ for each $j \in [3]$.  Finally, create the edges $(v_1, 1)(v_{2k+2},2)$, $(v_1, 2)(v_{2k+2},3)$, and $(v_1, 3)(v_{2k+2},1)$.  

We claim for each $j \in [3]$, $(w,j)$ is in exactly one $\mathcal{H}$-coloring of $M$.  We will prove this for $(w,1)$ (a similar argument holds for $(w,2)$ and $(w,3)$).  Suppose $I$ is an $\mathcal{H}$-coloring of $M$ containing $(w,1)$.  Then, $I - \{(w,1)\}$ is an independent set of size $2k+2$ in $H[\{(v_i,j): i \in [2k+2], j \in \{2,3\}\}$].  It is easy to verify that the only independent set of size $2k+2$ in $H[\{(v_i,j): i \in [2k+2], j \in \{2,3\}\}]$ is $\{(v_i,3): i \in \{1,3,\ldots, 2k+1\}\} \cup \{(v_i,2): i \in \{2,4,\ldots, 2k+2\}\}$.  Consequently, it must be that $I = \{(v_i,3): i \in \{1,3,\ldots, 2k+1\}\} \cup \{(v_i,2): i \in \{2,4,\ldots, 2k+2\}\} \cup \{(w,1)\}$ which means $(w,1)$ is in exactly one $\mathcal{H}$-coloring of $M$.  From this, it follows that $P_{DP}(M,3) \leq P_{DP}(M,\mathcal{H}) = 3$.  

To show that $P_{DP}(M,3) \geq 3$, we will use Lemma~\ref{lem: sethFormula}.  Notice that when $m = 3$ and $k \in \mathbb{N}$, $(m-2)-(m-2)^{2k+1} + ((m-2)^{2k+2} + (m-2))/{(m-1)} =
1 \geq 0$.
Thus,\\ 
$\min_{s \in \{0,1,3\})} \left(s(p - (s-1)p_1 - (3-s)p_2) + (3-s)(p - sp_1 - (3-2-s)p_2)\right)= 3(p - p_2) = 3.$
So, $P (M,3) \geq 3$.  Finally, $P_{DP}(M, m) = 0$ when $m \in [2]$ follows from the fact that $M$ contains a cycle.
\end{proof}

Notice that Theorem~\ref{thm: wheel} gives us the result in Theorem~\ref{thm: threshold} when $p=1$.  To complete the proof of Theorem~\ref{thm: threshold}, we will now show that $\tau_{DP}(K_p \vee C_{2k+2}) = 3 +p$ when $p \geq 2$ and $k \in \N$.

\begin{proof}
Let $M = K_p \lor G$, where $G = C_{2k+2}$ and $p \geq 2$.  Let the vertices in $M$ that correspond to $G$ be $v_1,v_2,\ldots,v_{2k+2}$ in cyclic order, and let the vertices corresponding to the copy of $K_p$ used to form $M$ be $w_1,w_2,\ldots,w_p$.  It is easy to see that $P(K_p\lor C_{2k+2},2+p)=(2+p)!$.  Also, a simple induction on $p$ that utilizes Theorems~\ref{thm: monoinp} and~\ref{thm: wheel} shows that $\tau_{DP}(K_p \lor G) \leq 3+p$.

We will now demonstrate $\tau_{DP}(M) > (2 + p)$ by showing $P_{DP}(M,2+p) < P(M,2+p) = (2+p)!$.  For each $v \in V(M)$, let $L(v) = \{(v,j): j \in [2+p]\}$.  Let $H$ be the graph with vertex set $\bigcup_{v \in V(M)} L(v)$.  We will now describe the construction of edges in $H$.  Create edges so that $H[L(v)]$ is a complete graph for each $v \in V(M)$.  Whenever $uv \in E(M) - \{v_1v_{2k+2}\}$, create an edge between $(u,j)$ and $(v,j)$ for each $j \in [2+p]$.  Finally, create the edges $(v_1, 1)(v_{2k+2},2)$, $(v_1, 2)(v_{2k+2},3), \ldots, (v_1, 1+p)(v_{2k+2}, 2+p),$ and $(v_1, 2+p)(v_{2k+2},1)$. Then $\mathcal{H}=(L,H)$ is a $(2+p)$-fold cover of $M$.

We claim that each selection $(w_1,j_1),(w_2,j_2),\ldots,(w_p,j_p)$ of nonadjacent vertices from $L(w_1), L(w_2), \ldots, L(w_p)$ respectively are in at most 2 $\mathcal{H}$-colorings of $M$.  To see this, consider such a selection $(w_1,j_1),(w_2,j_2),\ldots,(w_p,j_p)$ of nonadjacent vertices.  By construction, for each $i \in [p]$, $(w_i,j_i)$ is adjacent to $(v_n,j_i)$ for every $n\in[2k+2]$.  Then, the only remaining vertices in $\bigcup_{v \in V(G)} L(v)$ not adjacent to a chosen vertex are those with second coordinate in $[2+p]-\{j_1,j_2\ldots,j_p\}$.  This set has 2 elements, since $j_1,j_2,\ldots,j_p$ must be pairwise distinct.  Suppose $[2+p]-\{j_1,j_2\ldots,j_p\} = \{\alpha_1, \alpha_2 \}$.  Then, the set of vertices in $\bigcup_{v \in V(G)} L(v)$ nonadjacent to each $(w_1,j_1),(w_2,j_2),\ldots,(w_p,j_p)$ is $\{(v_j,\alpha_1) : j\in[2k+2]\}\cup\{(v_j,\alpha_2):j\in[2k+2]\}$.  Note $(v_j,\alpha_1)(v_{j+1},\alpha_1), (v_j,\alpha_2)(v_{j+1},\alpha_2) \in E(H)$ for every $j \in [2k+1]$, and $(v_j,\alpha_1)(v_j,\alpha_2) \in E(H)$ for every $j\in[2k+2]$.

So, there are clearly at most 2 ways to choose vertices from $\bigcup_{v \in V(G)} L(v)$ so as to complete an $\mathcal{H}$-coloring of $M$ (the 2 possible ways to do this are to choose the elements in $\{(v_1,\alpha_1),(v_2,\alpha_2),(v_3,\alpha_1),\ldots,(v_{2k+2},\alpha_2) \}$ or choose the elements in \\ $\{(v_1,\alpha_2),(v_2,\alpha_1),(v_3,\alpha_2),\ldots,(v_{2k+2},\alpha_1) \}$).

Notice that there are $(2+p)!/2$ ways to choose $p$ nonadjacent vertices from $\bigcup_{i=1}^p L(w_i)$.  So, if we can demonstrate that some selection of $p$ nonadjacent vertices from $\bigcup_{i=1}^p L(w_i)$ is in at most one $\mathcal{H}$-coloring of $M$, we will have $P_{DP}(G,\mathcal{H})<(2+p)!$ and our proof will be complete. 

Consider the selection $(w_1,1),(w_2,2),\ldots,(w_p,p)$ of nonadjacent vertices from \\ $L(w_1), L(w_2), \ldots, L(w_p)$ respectively.  Since $\{(v_1,p+1),(v_2,p+2),(v_3,p+1),\ldots,(v_{2k+2},p+2) \}$ is not an independent set in $H$, there is at most one $\mathcal{H}$-coloring that contains \\ $\{(w_1,1),(w_2,2),\ldots,(w_p,p) \}$.  This completes the proof.
\end{proof}

\section{Gluings of Graphs} \label{cliquesum}

Throughout this Section whenever we have an $m$-fold cover $\HH = (L,H)$ of some graph $G$, we assume that $L(x) = \{(x,j) : j \in [m]\}$ for each $x \in V(G)$ unless otherwise noted. 

\subsection{A Toolbox for Gluings of Graphs}\label{tools}

We begin by establishing some useful terminology and notation.  Whenever $\HH = (L,H)$ is an $m$-fold cover of $G$ and $S \subseteq V(H)$, we let $N(S,\HH)$ be the number of $\HH$-colorings containing $S \subseteq V(H)$. When $S = \{s\}$, we write $N(s,\HH)$.  The next two definitions are crucial for the results we present in this Section. First definition gives a natural way of splitting a cover of a $G \in \bigoplus_{i=1}^{n}(G_i,p)$ into corresponding covers for each $G_i$. The second definition shows how to combine covers of $G_i$ into a cover for a $G \in \bigoplus_{i=1}^{n}(G_i,1)$.

\begin{defn}\label{defn: separation}
For some $n \geq 2$, suppose that $G_1,\ldots,G_n$ are vertex disjoint graphs such that $\{u_{i,1},\ldots,u_{i,p}\}$ is a clique in $G_i$ for each $i \in [n]$ and some $p \in \N$. Let $G$ be the graph obtained by identifying $u_{1,q},\ldots,u_{n,q}$ as the same vertex $u_q$ for each $q \in [p]$. Suppose $\HH = (L,H)$ is an arbitrary $m$-fold cover of $G$.  For each $i \in [n]$, \textbf{the separated cover of $G_i$ obtained from $\HH$} is an $m$-fold cover $\HH_i = (L_i,H_i)$ of $G_i$ defined as follows. Let $L_i(x) = \{(x,j) : j \in [m]\}$ for each $x \in V(G_i)$. Construct edges of $H_i$ so that $H_i$ is isomorphic to $H_i' = H[(\bigcup_{x \in V(G_i) - \{u_{i,q} : q \in [p]\}} L(x)) \cup (\bigcup_{q=1}^{p} L(u_q))]$ and $f: V(H_i') \rightarrow V(H_i)$ given by
\[ f((x,j)) = \begin{cases}
    (x,j) & \text{if } x \in V(G_i) - \{u_{i,q} : q \in [p]\}\\
    (u_{i,q},j) & \text{if } x = u_q
   \end{cases}
\]
is a graph isomorphism.
\end{defn}

\begin{defn}\label{defn: vertex amalgamation}
Suppose that $G_1,\ldots,G_n$ are vertex disjoint graphs graphs for some $n \geq 2$. Suppose that $m \in \N$ and that $\HH_i = (L_i,H_i)$ is an $m$-fold cover of $G_i$. Suppose that $u_i \in V(G_i)$ for each $i \in [n]$ and $f_{k+1}: L_1(u_1) \rightarrow L_{k+1}(u_{k+1})$ is a bijection for each $k \in [n-1]$. Let $F = (f_2,\ldots,f_n)$, and let $G$ be the graph obtained by identifying $u_1,\ldots,u_n$ as the same vertex $u$. \textbf{The $F$-amalgamated $m$-fold cover of $G$ obtained from $\HH_1,\ldots,\HH_n$} is an $m$-fold cover $\HH = (L,H)$ of $G$ defined as follows. (In the special case where $n = 2$, we may also say $f_2$-amalgamated $m$-fold cover of $G$ obtained from $\HH_1$ and $\HH_2$.)  For each $i \in [n]$, assume $L_i(x) = \{(x,j): j \in [m]\}$ for each $x \in V(G_i)$.  Let $L(x) = \{(x,j): j \in [m]\}$ for each $x \in V(G)$. Let $X_i$ be the set of edges in $H_i$ that are incident to at least one element in $L_i(u_i)$. Construct edges in $H$ so that $H[L(u)]$ is a complete graph and $H$ contains the edges in $\bigcup_{i=1}^{n} (E(H_i) -  X_i)$. Then for each $s \in [m]$, whenever $(x,r)(u_1,s) \in E(H_1)$ where $(x,r) \notin L_1(u_1)$, construct the edge $(x,r)(u,s)$ in $H$. Finally, for each $s \in [m]$ and $i \in [n-1]$, whenever $(x,r)f_{i+1}((u_1,s)) \in E(H_{i+1})$ where $(x,r) \notin L_{i+1}(u_{i+1})$, construct the edge $(x,r)(u,s)$ in $H$.
\end{defn}
Below is an illustration of Definition~\ref{defn: vertex amalgamation} where $n=m=2$, $G_1= P_3$, $G_2 = P_3$, $f_2((u_1,1))=(u_2,2)$, and $f_2((u_1,2))=(u_2,1)$.
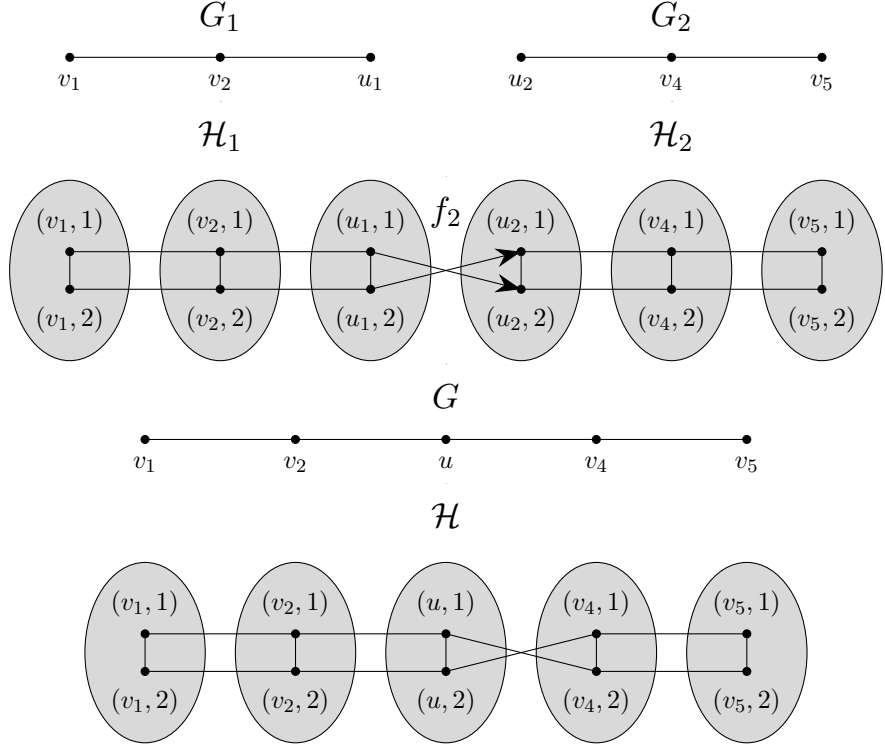
\begin{center}
\begin{tikzpicture}
    \coordinate (v1) at (-5,0);
    \coordinate (v2) at (-3,0);
    \coordinate (v7) at (-3,1);
    \coordinate (v3) at (-1,0);
    \coordinate (v4) at (1,0);
    \coordinate (v5) at (3,0);
    \coordinate (v8) at (3,1);
    \coordinate (v6) at (5,0);
    
    \draw[fill=black] (v1) circle[radius=1.5pt] node[below=3pt,scale=0.9] {$v_1$};
    \draw[fill=black] (v2) circle[radius=1.5pt] node[below=3pt,scale=0.9] {$v_2$};
    \draw[fill=black] (v7) circle[radius=0pt] node[below=3pt,scale=1.2] {$G_1$};
    \draw[fill=black] (v3) circle[radius=1.5pt] node[below=3pt,scale=0.9] {$u_1$};
    \draw[fill=black] (v4) circle[radius=1.5pt] node[below=3pt,scale=0.9] {$u_2$};
    \draw[fill=black] (v8) circle[radius=0pt] node[below=3pt,scale=1.2] {$G_2$};
    \draw[fill=black] (v5) circle[radius=1.5pt] node[below=3pt,scale=0.9] {$v_4$};
    \draw[fill=black] (v6) circle[radius=1.5pt] node[below=3pt,scale=0.9] {$v_5$};
    
    \draw (v1) -- (v2);
    \draw (v2) -- (v3);
    \draw (v4) -- (v5);
    \draw (v5) -- (v6);
\end{tikzpicture}

\begin{tikzpicture}
    \coordinate (v11) at (-5,0);
    \coordinate (v12) at (-5,0.5);
    \coordinate (v21) at (-3,0);
    \coordinate (v22) at (-3,0.5);
    \coordinate (v7) at (-3,2.5);
    \coordinate (v31) at (-1,0);
    \coordinate (v32) at (-1,0.5);
    \coordinate (v9) at (0,1.5);
    \coordinate (v41) at (1,0);
    \coordinate (v42) at (1,0.5);
    \coordinate (v51) at (3,0);
    \coordinate (v52) at (3,0.5);
    \coordinate (v8) at (3,2.5);
    \coordinate (v61) at (5,0);
    \coordinate (v62) at (5,0.5);
    
    \draw[fill=gray!30] (-5,0.25) ellipse (8mm and 12mm);
    \draw[fill=gray!30] (-3,0.25) ellipse (8mm and 12mm);
    \draw[fill=gray!30] (-1,0.25) ellipse (8mm and 12mm);
    \draw[fill=gray!30] (1,0.25) ellipse (8mm and 12mm);
    \draw[fill=gray!30] (3,0.25) ellipse (8mm and 12mm);
    \draw[fill=gray!30] (5,0.25) ellipse (8mm and 12mm);
    
    \draw[fill=black] (v12) circle[radius=1.5pt] node[above=3pt,scale=0.9] {$(v_1,1)$};
    \draw[fill=black] (v11) circle[radius=1.5pt] node[below=3pt,scale=0.9] {$(v_1,2)$};
    \draw[fill=black] (v22) circle[radius=1.5pt] node[above=3pt,scale=0.9] {$(v_2,1)$};
    \draw[fill=black] (v21) circle[radius=1.5pt] node[below=3pt,scale=0.9] {$(v_2,2)$};
    \draw[fill=black] (v7) circle[radius=0pt] node[below=3pt,scale=1.2] {$\HH_1$};
    \draw[fill=black] (v32) circle[radius=1.5pt] node[above=3pt,scale=0.9] {$(u_1,1)$};
    \draw[fill=black] (v31) circle[radius=1.5pt] node[below=3pt,scale=0.9] {$(u_1,2)$};
    \draw[fill=black] (v9) circle[radius=0pt] node[below=3pt,scale=1.2] {$f_2$};
    \draw[fill=black] (v42) circle[radius=1.5pt] node[above=3pt,scale=0.9] {$(u_2,1)$};
    \draw[fill=black] (v41) circle[radius=1.5pt] node[below=3pt,scale=0.9] {$(u_2,2)$};
    \draw[fill=black] (v52) circle[radius=1.5pt] node[above=3pt,scale=0.9] {$(v_4,1)$};
    \draw[fill=black] (v51) circle[radius=1.5pt] node[below=3pt,scale=0.9] {$(v_4,2)$};
    \draw[fill=black] (v8) circle[radius=0pt] node[below=3pt,scale=1.2] {$\HH_2$};
    \draw[fill=black] (v62) circle[radius=1.5pt] node[above=3pt,scale=0.9] {$(v_5,1)$};
    \draw[fill=black] (v61) circle[radius=1.5pt] node[below=3pt,scale=0.9] {$(v_5,2)$};
    
    \draw (v11) -- (v12);
    \draw (v21) -- (v22);
    \draw (v31) -- (v32);
    \draw (v41) -- (v42);
    \draw (v51) -- (v52);
    \draw (v61) -- (v62);
    
    \draw (v11) -- (v21);
    \draw (v12) -- (v22);
    \draw (v21) -- (v31);
    \draw (v22) -- (v32);
    \draw[-{Stealth[scale=2.0]}] (v31) -- (v42);
    \draw[-{Stealth[scale=2.0]}] (v32) -- (v41);
    \draw (v41) -- (v51);
    \draw (v42) -- (v52);
    \draw (v51) -- (v61);
    \draw (v52) -- (v62);

\end{tikzpicture}

\begin{tikzpicture}
    \coordinate (v1) at (-4,0);
    \coordinate (v2) at (-2,0);
    \coordinate (v3) at (0,0);
    \coordinate (v6) at (0,1);
    \coordinate (v4) at (2,0);
    \coordinate (v5) at (4,0);
    
    \draw[fill=black] (v1) circle[radius=1.5pt] node[below=3pt,scale=0.9] {$v_1$};
    \draw[fill=black] (v2) circle[radius=1.5pt] node[below=3pt,scale=0.9] {$v_2$};
    \draw[fill=black] (v3) circle[radius=1.5pt] node[below=3pt,scale=0.9] {$u$};
    \draw[fill=black] (v6) circle[radius=0pt] node[below=3pt,scale=1.2] {$G$};
    \draw[fill=black] (v4) circle[radius=1.5pt] node[below=3pt,scale=0.9] {$v_4$};
    \draw[fill=black] (v5) circle[radius=1.5pt] node[below=3pt,scale=0.9] {$v_5$};
    
    \draw (v1) -- (v2);
    \draw (v2) -- (v3);
    \draw (v3) -- (v4);
    \draw (v4) -- (v5);
\end{tikzpicture}

\begin{tikzpicture}
    \coordinate (v11) at (-3,0);
    \coordinate (v12) at (-3,0.5);
    \coordinate (v21) at (-1,0);
    \coordinate (v22) at (-1,0.5);
    \coordinate (v31) at (1,0);
    \coordinate (v32) at (1,0.5);
    \coordinate (v6) at (1,2.5);
    \coordinate (v41) at (3,0);
    \coordinate (v42) at (3,0.5);
    \coordinate (v51) at (5,0);
    \coordinate (v52) at (5,0.5);
    
    \draw[fill=gray!30] (-3,0.25) ellipse (8mm and 12mm);
    \draw[fill=gray!30] (-1,0.25) ellipse (8mm and 12mm);
    \draw[fill=gray!30] (1,0.25) ellipse (8mm and 12mm);
    \draw[fill=gray!30] (3,0.25) ellipse (8mm and 12mm);
    \draw[fill=gray!30] (5,0.25) ellipse (8mm and 12mm);
    
    \draw[fill=black] (v12) circle[radius=1.5pt] node[above=3pt,scale=0.9] {$(v_1,1)$};
    \draw[fill=black] (v11) circle[radius=1.5pt] node[below=3pt,scale=0.9] {$(v_1,2)$};
    \draw[fill=black] (v22) circle[radius=1.5pt] node[above=3pt,scale=0.9] {$(v_2,1)$};
    \draw[fill=black] (v21) circle[radius=1.5pt] node[below=3pt,scale=0.9] {$(v_2,2)$};
    \draw[fill=black] (v32) circle[radius=1.5pt] node[above=3pt,scale=0.9] {$(u,1)$};
    \draw[fill=black] (v31) circle[radius=1.5pt] node[below=3pt,scale=0.9] {$(u,2)$};
    \draw[fill=black] (v6) circle[radius=0pt] node[below=3pt,scale=1.2] {$\HH$};
    \draw[fill=black] (v42) circle[radius=1.5pt] node[above=3pt,scale=0.9] {$(v_4,1)$};
    \draw[fill=black] (v41) circle[radius=1.5pt] node[below=3pt,scale=0.9] {$(v_4,2)$};
    \draw[fill=black] (v52) circle[radius=1.5pt] node[above=3pt,scale=0.9] {$(v_5,1)$};
    \draw[fill=black] (v51) circle[radius=1.5pt] node[below=3pt,scale=0.9] {$(v_5,2)$};
    
    \draw (v11) -- (v12);
    \draw (v21) -- (v22);
    \draw (v31) -- (v32);
    \draw (v41) -- (v42);
    \draw (v51) -- (v52);
    
    \draw (v11) -- (v21);
    \draw (v12) -- (v22);
    \draw (v21) -- (v31);
    \draw (v22) -- (v32);
    \draw (v31) -- (v42);
    \draw (v32) -- (v41);
    \draw (v41) -- (v51);
    \draw (v42) -- (v52);
\end{tikzpicture}
\end{center}

Definition~\ref{defn: vertex amalgamation} allows us to present a result which gives a general method for finding an upper bound on the DP color function of a vertex-gluing of graphs.

\begin{lem}\label{lem: upperGen}
Suppose that $G_1,\ldots,G_n$ are vertex disjoint graphs for some $n \geq 2$ where $u_i \in V(G_i)$ for each $i \in [n]$.  Suppose that $\HH_i = (L_i,H_i)$ is an $m$-fold cover of $G_i$ for each $i \in [n]$. Suppose $f_i: L_1(u_1) \rightarrow L_i(u_i)$ is a bijection for each $2 \leq i \leq n$, and let $F = (f_2,\ldots,f_n)$. Let $G$ be the graph obtained by identifying $u_1,\ldots,u_n$ as the same vertex $u$. Let $D = \sum_{j=1}^m (N((u_1,j),\HH_1) \prod_{i=2}^{n} N(f_i((u_1,j)),\HH_i))$. Then $P_{DP}(G,\HH) = D$ where $\HH$ is the $F$-amalgamated cover of $G$ obtained from $\HH_1,\ldots,\HH_n$. Consequently, $P_{DP}(G,m) \leq D$. 
\end{lem}
\begin{proof}
Let $\HH = (L,H)$ be the $F$-amalgamated $m$-fold cover of $G$ obtained from $\HH_1,\ldots,\HH_n$. We will count the number of $\HH$-colorings of $G$ that contain $(u,j)$ for some $j \in [m]$. Note $P_{DP}(G,\HH) = \sum_{j=1}^m N((u,j),\HH)$. Let $\II_1$ be the set of all $\HH_1$-colorings of $G_1$ that contain $(u_1,j)$, $\II_i$ be the set of all $\HH_i$-colorings of $G_i$ that contain $f_i((u_1,j))$ for each $2 \leq i \leq n$, and $\II$ be the set of all $\HH$-colorings of $G$ that contain $(u,j)$ for some $j \in [m]$. Let $g: \II_1 \times \cdots \times \II_n \rightarrow \II$ be the function given by $g((I_1,\ldots,I_n)) = \{(u,j)\} \cup (I_1 - \{(u_1,j)\}) \cup \bigcup_{i=2}^{n} (I_i - \{f_i((u_1,j))\})$. It is easy to check that $g((I_1,\ldots,I_n))$ is an independent set of size $|V(G)|$ in $H$. Also, $g$ is a bijection. As such, $N((u,j),\HH) = |\II| = \prod_{i=1}^{n} |\II_i| = N((u_1,j),\HH_1)\prod_{i=2}^{n} N(f_i((u_1,j)),\HH_i)$ for each $j \in [m]$. Therefore, $P_{DP}(G,\HH) = D$ which implies that $P_{DP}(G,m) \leq D$.
\end{proof}

Lemma~\ref{lem: upperGen} allows us to prove Theorem~\ref{thm: upperbound} which shows that the inequality in Question~\ref{ques: classify} holds in the case that $p=1$.  We now restate Theorem~\ref{thm: upperbound}.  

\begin{customthm} {\ref{thm: upperbound}}
Suppose that $G_1,\ldots,G_n$ are vertex disjoint graphs for some $n \geq 2$ and $G \in \bigoplus_{i=1}^{n}(G_i,1)$. Then $$P_{DP}(G,m) \leq \frac{\prod_{i=1}^{n} P_{DP}(G_i,m)}{m^{n-1}}.$$ 
\end{customthm}

\begin{proof}
Throughout this argument suppose $u_i \in V(G_i)$ for each $i \in [n]$ and $G$ is the graph obtained by identifying $u_1,\ldots,u_n$ as the same vertex $u$.

The proof is by induction on $n$. We begin by proving the result for $n = 2$. Suppose $\HH_i = (L_i,H_i)$ is an $m$-fold cover of $G_i$ such that $P_{DP}(G_i,\HH_i) = P_{DP}(G_i,m)$ and $L_i(u_i) = \{(u_i,j) : j \in \{0\} \cup [m-1]\}$ for each $i \in [2]$. For each $0 \leq j \leq m-1$, let $a_j = N((u_1,j),\HH_1)$ and $b_{j} = N((u_2,j),\HH_2)$.  For each $d \in \{0\} \cup [m-1]$ let $f_d: L_1(u_1) \rightarrow L_2(u_2)$ be defined by $f_d((u_1,j)) = (u_2,j+d)$ where addition is performed mod $m$. Notice $f_d$ is a bijection. For the rest of the argument, whenever we have $b_{j+d}$, addition is performed mod $m$. Let $D_{f_d} = \sum_{j=0}^{m-1} a_jb_{j+d}$. By Lemma~\ref{lem: upperGen}, we know that $P_{DP}(G,m) \leq D_{f_d}$ for each $0 \leq d \leq m-1$.  Notice $P_{DP}(G_1,m) = \sum_{j=0}^{m-1} a_j$ and $P_{DP}(G_2,m) = \sum_{d=0}^{m-1} b_{j+d}$ for each $0 \leq j \leq m-1$. Thus, $$\sum_{d=0}^{m-1} D_{f_d} = \sum_{d=0}^{m-1} \sum_{j=0}^{m-1} a_jb_{j+d} = \sum_{j=0}^{m-1} \sum_{d=0}^{m-1} a_jb_{j+d} = \sum_{j=0}^{m-1} a_j \sum_{d=0}^{m-1} b_{j+d} = P_{DP}(G_1,m)P_{DP}(G_2,m).$$ For some $c \in \{0\} \cup [m-1]$, $$P_{DP}(G,m) \leq D_{f_c} \leq \frac{\sum_{d=0}^{m-1} D_{f_d}}{m} = \frac{P_{DP}(G_1,m)P_{DP}(G_2,m)}{m}.$$

Now suppose $n \geq 3$ and the result holds for all natural numbers greater than $1$ and less than $n$. Let $G'$ be the graph obtained by identifying $u_1,\ldots,u_{n-1}$ as the same vertex $u'$. By the inductive hypothesis, for each $m \in \N$, $$P_{DP}(G',m) \leq \frac{\prod_{i=1}^{n-1} P_{DP}(G_i,m)}{m^{n-2}}.$$ Notice $G$ is the graph obtained by identifying $u'$ and $u_n$ as the same vertex $u$. Thus, by the inductive hypothesis, $$P_{DP}(G,m) \leq \frac{P_{DP}(G',m)P_{DP}(G_n,m)}{m} \leq \frac{\prod_{i=1}^{n} P_{DP}(G_i,m)}{m^{n-1}}.$$
\end{proof}

Using Definition~\ref{defn: separation} we are now ready to prove Lemma~\ref{lem: lowerGen} which gives us a general method for finding a lower bound on the DP color function of $K_p$-gluings of graphs.  We first restate Lemma~\ref{lem: lowerGen} using the notation we have established in this Section.

\begin{customlem} {\ref{lem: lowerGen}}
Suppose that $G_1,\ldots,G_n$ are vertex disjoint graphs where $n \geq 2$ and $G \in \bigoplus_{i=1}^n(G_i,p)$ where for each $i \in [n]$ $\{u_{i,1},\ldots,u_{i,p}\}$ is a clique in $G_i$ and $G$ is obtained by identifying $u_{1,q},\ldots,u_{n,q}$ as the same vertex $u_q$ for each $q \in [p]$. Suppose that for each $i \in [n]$, given any $m$-fold cover $\DD_i = (K_i,D_i)$ of $G_i$, $N(A,\DD_i) \geq k_i$ whenever $A \subseteq \bigcup_{q=1}^{p} K_i(u_{i,q})$, $|A \cap K_i(u_{i,q})| = 1$ for each $q \in [p]$, and $A$ is an independent set in $D_i$. Then $$P_{DP}(G,m) \geq \left(\prod_{i=0}^{p-1} (m-i)\right) \left(\prod_{i=1}^{n} k_i\right).$$
\end{customlem}

\begin{proof}
Suppose that $\HH = (L,H)$ is an arbitrary $m$-fold cover of $G$. For each $i \in [n]$, assume $\HH_i = (L_i,H_i)$ is the separated cover of $G_i$ obtained from $\HH$. Now we will count the number of $\HH$-colorings of $G$ that contain $(u_1,s_1),\ldots,(u_p,s_p) \in V(H)$ where the set containing those elements is an independent set in $H$. For each $i \in [n]$, let $\II_i$ be the set of all $\HH_i$-colorings of $G_i$ that contain $(u_{i,1},s_1),\ldots,(u_{i,p},s_p)$. Let $\II$ be the set of all $\HH$-colorings of $G$ that contain $(u_1,s_1),\ldots,(u_p,s_p)$. Let $P = \{(u_q,s_q) : q \in [p]\}$, and for each $i \in [n]$, let $P_i = \{(u_{i,q},s_q) : q \in [p]\}$. Let $g: \II_1 \times \ldots \times \II_n \rightarrow \II$ be the function given by $g((I_1,\ldots,I_n)) = P \cup \bigcup_{i=1}^{n} (I_i - P_i)$. It is easy to check that $P \cup \bigcup_{i=1}^{n} (I_i - P_i)$ is an independent set of size $|V(G)|$ in $H$. Also, $g$ is a bijection. Thus, $N(P,\HH) = |\II| = \prod_{i=1}^{n} |\II_i| = \prod_{i=1}^{n} N(P_i,\HH_i) \geq \prod_{i=1}^{n} k_i$. Since there are at least $\prod_{i=0}^{p-1} (m-i)$ different ways to form an independent set in $H$ by selecting exactly one element from each of $L(u_1), \ldots, L(u_p)$, $P_{DP}(G,\HH) \geq (\prod_{i=0}^{p-1} (m-i)) (\prod_{i=1}^{n} k_i)$. Since $\HH$ is arbitrary, $P_{DP}(G,m) \geq (\prod_{i=0}^{p-1} (m-i)) (\prod_{i=1}^{n} k_i)$.
\end{proof}

The next two results follow from Theorem~\ref{thm: upperbound} and Lemma~\ref{lem: lowerGen}, and they are helpful for making progress on Question~\ref{ques: classify}.

\begin{cor}\label{cor: lowerGenPretty}
Suppose that $G_1,\ldots,G_n$ are vertex disjoint graphs where $n \geq 2$ and $G \in \bigoplus_{i=1}^n(G_i,p)$ where for each $i \in [n]$, $\{u_{i,1},\ldots,u_{i,p}\}$ is a clique in $G_i$ and $G$ is obtained by identifying $u_{1,q},\ldots,u_{n,q}$ as the same vertex $u_q$ for each $q \in [p]$. Suppose that for each $i \in [n]$, given any $m$-fold cover $\DD_i = (K_i,D_i)$ of $G_i$, $N(A,\DD_i) \geq P_{DP}(G_i,m) / \prod_{i=0}^{p-1} (m-i)$ whenever $A \subseteq \bigcup_{q=1}^{p} L(u_{i,q})$, $|A \cap L(u_{i,q})| = 1$ for each $q \in [p]$, and $A$ is an independent set in $D_i$. Then $$P_{DP}(G,m) \geq \frac{\prod_{i=1}^{n} P_{DP}(G_i,m)}{\left(\prod_{i=0}^{p-1} (m-i)\right)^{n-1}}.$$
\end{cor}

\begin{cor}\label{thm: general}
Suppose that $G_1,\ldots,G_n$ are vertex disjoint graphs where $n \geq 2$ and $G \in \bigoplus_{i=1}^n(G_i,1)$ where $u_i \in V(G_i)$ for each $i \in [n]$ and $G$ is the graph obtained by identifying $u_1,\ldots,u_n$ as the same vertex $u$. Also suppose that for each $i \in [n]$ and any $m$-fold cover $\HH_i=(L_i,H_i)$ of $G_i$, $N(s,\HH_i) \geq P_{DP}(G_i,m)/m$ for each $s \in L_i(u_i)$. Then $$P_{DP}(G,m) = \frac{\prod_{i=1}^{n} P_{DP}(G_i,m)}{m^{n-1}}.$$
\end{cor}

\subsection{Vertex-gluings of Cycles and Chordal Graphs}\label{CycleChordal}

We are now ready to work towards a proof of Theorem~\ref{thm: ChordCycle} which is an application of Corollary~\ref{thm: general}.  To prove Theorem~\ref{thm: ChordCycle} we need only show that cycles and chordal graphs satisfy the hypotheses of Corollary~\ref{thm: general}.  We begin by showing the cycles satisfy the hypotheses of Corollary~\ref{thm: general}.  We first need a result from~\cite{KM19}.  Recall that a \emph{unicyclic graph} is a connected graph containing exactly one cycle.

\begin{thm} [\cite{KM19}] \label{thm: onecycle}
Let $G$ be a unicyclic graph on $n$ vertices.\\
(i) For $m \in \N$, if $G$ contains a cycle on $2k+1$ vertices, then $P_{DP}(G,m) = P(G,m) = (m-1)^n - (m-1)^{n-2k}$.
\\
(ii)  For $m \geq 2$, if $G$ contains a cycle on $2k+2$ vertices, then \\ $P_{DP}(G,m) = (m-1)^n - (m-1)^{n-2k-2}.$
\end{thm}

\begin{lem}\label{lem: lowerCyc}
Suppose that $G = C_l$ where $l \geq 3$. Also suppose that $\HH = (L,H)$ is an arbitrary $m$-fold cover of $G$ where $m \geq 2$. For each $v \in V(G)$, whenever $r \in L(v)$, $N(r,\HH) \geq P_{DP}(G,m)/m$.
\end{lem}
\begin{proof}
Suppose that the vertices of $G$ in cyclic order are $s_1,s_2,\ldots , s_n$. We can assume that $\HH$ is full since adding edges can not increase the number of independent sets of a graph with a prescribed size that contain a given vertex.  Let $H' = H - E_H(L(s_1),L(s_n))$. Notice $\HH' = (L,H')$ is an $m$-fold cover of $G'$ where $G' = G - \{s_1s_n\}$. Since $G'$ is a tree, Proposition~\ref{pro: tree} implies $\HH'$ has a canonical labeling.  So, we can assume for each $i \in [n]$, $L(s_i) = \{(s_i,j): j \in [m]\}$, and for each $l \in [n-1]$, $(s_l,j)(s_{l+1},j) \in E(H')$ whenever $j \in [m]$. Let $L''$ be $L$ with its domain restricted to $\{s_2,\ldots , s_n\}$. Let $H'' = H - L(s_1)$. Notice $\HH'' = (L'',H'')$ is an $m$-fold cover of $G''$ where $G'' = G - \{s_1\}$. Without loss of generality, we assume that $r = (s_1,1)$. We know that $(s_2,1) \in N_H((s_1,1))$. Suppose the unique element in $L(s_n) \cap N_H((s_1,1))$ is $(s_n,q)$.

When $n=3$ we consider two cases: (1) $q=1$ and (2) $q \neq 1$. In case (1), $N(r,\HH)$ is the number of $\HH''$-colorings of $G''$ that are disjoint from $\{(s_2,1),(s_3,1)\}$. This implies $N(r,\HH)$ is the number of proper $m$-colorings $f$ of $G''$ satisfying: $s_2 \notin f^{-1}(1)$ and $s_3 \notin f^{-1}(1)$. It is then easy to see that: $N(r,\HH)=(m-1)(m-2) = P_{DP}(G,m)/m$.

In case (2), $N(r,\HH)$ is the number of $\HH''$-colorings of $G''$ that are disjoint from $\{(s_2,1),(s_3,q)\}$. This implies $N(r,\HH)$ is the number of proper $m$-colorings $f$ of $G''$ satisfying: $s_2 \notin f^{-1}(1)$ and $s_3 \notin f^{-1}(q)$. It is then easy to see that: $N(r,\HH) > (m-1)(m-2) = P_{DP}(G,m)/m$.

We will now prove the statement when $n \geq 4$. We will consider two cases: (1) $n$ is even and (2) $n$ is odd. In case (1), $n = 2k + 2$ for some $k \in \N$. In the subcase where $q = 1$, $N(r,\HH)$ is the number of $\HH''$-colorings of $G''$ that are disjoint from $\{(s_2,1),(s_n,1)\}$. This implies $N(r,\HH)$ is the number of proper $m$-colorings $f$ of $G''$ satisfying: $s_2 \notin f^{-1}(1)$ and $s_n \notin f^{-1}(1)$. By the inclusion-exclusion principle we compute~\footnote{For this paper, we will assume that $C_2 = P_2$. Notice that $P_{DP}(C_2,m) = m(m-1)$.}, 
\begin{align*}
N(r,\HH) &= P(P_{2k+1},m) - \left(\frac{P(P_{2k+1},m)}{m} + \frac{P(P_{2k+1},m)}{m} - \frac{P(C_{2k},m)}{m}\right)\\
&= m(m-1)^{2k} - \left(2(m-1)^{2k} - \frac{(m-1)^{2k} +(m-1)}{m}\right)\\
&=\frac{(m-1)^{2k+2}+(m-1)}{m}> \frac{P_{DP}(G,m)}{m}.
\end{align*}

In the subcase where $q \neq 1$, $N(r,\HH)$ is the number of $\HH''$-colorings of $G''$ that are disjoint from $\{(s_2,1),(s_n,q)\}$. This implies $N(r,\HH)$ is the number of proper $m$-colorings $f$ of $G''$ satisfying: $s_2 \notin f^{-1}(1)$ and $s_n \notin f^{-1}(q)$. By the inclusion-exclusion principle we compute, 
\begin{align*}
N(r,\HH) &= P(P_{2k+1},m) - \left(\frac{P(P_{2k+1},m)}{m} + \frac{P(P_{2k+1},m)}{m} - \frac{P(C_{2k+1},m)}{m(m-1)}\right)\\
&= m(m-1)^{2k} - \left(2(m-1)^{2k} - \frac{(m-1)^{2k+1} - (m-1)}{m(m-1)}\right)\\
&=\frac{(m-1)^{2k+2}-1}{m}=\frac{P_{DP}(G,m)}{m}
\end{align*}

In case (2), $n = 2k + 3$ for some $k \in \N$. In the subcase where $q = 1$, $N(r,\HH)$ is the number of $\HH''$-colorings of $G''$ that are disjoint from $\{(s_2,1),(s_n,1)\}$. This implies $N(r,\HH)$ is the number of proper $m$-colorings $f$ of $G''$ satisfying: $s_2 \notin f^{-1}(1)$ and $s_n \notin f^{-1}(1)$. By the inclusion-exclusion principle we compute,
\begin{align*}
N(r,\HH) &= P(P_{2k+2},m) - \left(\frac{P(P_{2k+2},m)}{m} + \frac{P(P_{2k+2},m)}{m} - \frac{P(C_{2k+1},m)}{m}\right)\\
&= m(m-1)^{2k + 1} - \left(2(m-1)^{2k + 1} - \frac{(m-1)^{2k + 1} - (m-1)}{m}\right) \\
&=\frac{(m-1)^{2k+3}-(m-1)}{m}=\frac{P_{DP}(G,m)}{m}.
\end{align*}

In the subcase where $q \neq 1$, $N(r,\HH)$ is the number of $\HH''$-colorings of $G''$ that are disjoint from $\{(s_2,1),(s_n,q)\}$. This implies $N(r,\HH)$ is the number of proper $m$-colorings $f$ of $G''$ satisfying: $s_2 \notin f^{-1}(1)$ and $s_n \notin f^{-1}(q)$. By the inclusion-exclusion principle we compute,
\begin{align*}
N(r,\HH) &= P(P_{2k+2},m) - \left(\frac{P(P_{2k+2},m)}{m} + \frac{P(P_{2k+2},m)}{m} - \frac{P(C_{2k+2},m)}{m(m-1)}\right)\\
&= m(m-1)^{2k + 1} - \left(2(m-1)^{2k + 1} - \frac{(m-1)^{2k+2}+(m-1)}{m(m-1)}\right) \\
&=\frac{(m-1)^{2k+3}+1}{m}> \frac{P_{DP}(G,m)}{m}.
\end{align*}
This completes the proof. \end{proof}

As we turn our attention to chordal graphs, we recall one fact.  A \emph{perfect elimination ordering} for a graph $G$ is an ordering of the elements of $V(G)$, $v_1, v_2, \ldots, v_n$, such that for each vertex $v_i$, the neighbors of $v_i$ that occur after $v_i$ in the ordering form  a clique in $G$.  It is well known that a graph $G$ is chordal if and only if there is a perfect elimination ordering for $G$~\cite{FG65}.

\begin{lem}\label{lem: lowerChord}
Suppose that $G$ is a chordal graph with $n$ vertices. Also suppose $m \geq \chi(G)$ and that $\HH = (L,H)$ is an arbitrary $m$-fold cover of $G$. For each $u \in V(G)$, whenever $r \in L(u)$, $N(r,\HH) \geq P_{DP}(G,m)/m$.
\end{lem}
\begin{proof}
We are able to construct a perfect elimination ordering $v_1,\ldots,v_n$ where $v_n = u$ (see~\cite{D61}). We let $\alpha_i$ denote the number of neighbors of $v_i$ that occur after $v_i$ in the ordering. Note $\prod_{i=1}^{n} (m - \alpha_i) = P_{DP}(G,m)$ (see~\cite{KM19}), $\alpha_n = 0$, and $\chi (G) = 1 + \max_{i \in [n]}(\alpha_i)$. We will now inductively construct an $\HH$-coloring of $G$. Let $a_n = r$. For each $1 \leq i \leq n-1$ and any $0 \leq j < i$, if $v_{n-i}v_{n-j} \in E(G)$, then there is at most one vertex in $L(v_{n-i})$ that is adjacent to $a_{n-j}$ in $H$. Similarly for each $1 \leq i \leq n-1$ and any $0 \leq j < i$, if $v_{n-i}v_{n-j} \notin E(G)$, then there are no vertices in $L(v_{n-i})$ that are adjacent to $a_{n-j}$ in $H$. Choose some element $a_{n-i}$ from $L(v_{n-i})$ such that $a_{n-i}$ is not adjacent to any element in $\{a_{n-i+1},\ldots,a_n\}$. Clearly, there are at least $m-\alpha_{n-i} \geq 1$ choices for $a_{n-i}$. Thus, $N(r,\HH) \geq \prod_{i=1}^{n-1}(m-\alpha_{n-i}) = \prod_{i=1}^{n-1} (m - \alpha_{i}) = \prod_{i=1}^{n} (m - \alpha_{i})/m = P_{DP}(G,m)/m.$
\end{proof}

Now, Theorem~\ref{thm: ChordCycle} immediately follows from Lemma~\ref{lem: lowerCyc}, Lemma~\ref{lem: lowerChord}, and Corollary~\ref{thm: general}.

\subsection{Cones of the Disjoint Union of Cycles}\label{ConeCycles}

We will now finish the paper by studying cones of the disjoint unions of cycles. We will again utilize the tools from Section~\ref{tools} by thinking of such cones as vertex-gluings of wheels where we identify as the same vertex the universal vertices from each of the wheels. 

Let us establish some notation that we will use throughout the section. Let $G$ be the disjoint union of cycles $C_{k_i}$ for $i\in[n]$, with each $k_i \ge 3$. For each $i \in[n]$, let $M_i$ be a copy of $K_1 \vee C_{k_i}$ where $w_i$ is the universal vertex of $M_i$. Let $M = K_1 \vee G$. We will think of $M$ as the graph obtained from vertex-gluing all $M_i$ by identifying $w_1,\ldots,w_n$ as the same vertex $w$ in $M$.

Our aim is to prove the following. 

\begin{thm}\label{cor: coneofcycles}
	Let $M = K_1 \vee G$, where $G$ is the disjoint union of cycles $C_{k_i}$ for $i\in[n]$, with each $k_i \ge 3$. Then,
	\[
	\tau_{DP}(M) =
	\begin{cases}
		5 & \text{if there exist distinct $i,j \in [n]$ such that $k_i=k_j=4$} \\
		4 & \text{otherwise}. 
	\end{cases}
	\]
\end{thm}

This result demonstrates that the converse of Corollary~\ref{thm: general} does not hold, and it makes progress on Questions~\ref{ques: mono}, \ref{ques: join}, and~\ref{ques: classify} for these graphs.

We begin with a useful definition, observation, and lemma from~\cite{KM21}.  Suppose $G$ is a graph and $\HH = (L,H)$ is an $m$-fold cover of $G$. We say $\HH$ has a \emph{twisted canonical labeling} if $\HH$ is a full $m$-fold cover of $G$ and it is possible to let $L(x) = \{(x,j) : j \in [m]\}$ for each $x \in V(G)$ and choose two adjacent vertices, $u$ and $v$  in $G$ so that whenever $xy \in E(G)-\{uv\}$, $(x,j)$ and $(y,j)$ are adjacent in $H$ for each $j \in [m]$ and there exists $l \in [m]$ such that $(u,l)(v,l) \notin E(H)$. We call the matching $E_{H}(L(u),L(v))$ the \emph{twist}.

\begin{obs} [\cite{KM21}] \label{obsrv: labeling}
Suppose $G$ is a cycle and $\HH = (L,H)$ is an $m$-fold cover of $G$. $\HH$ has a twisted canonical labeling if and only if $\HH$ is a full cover and does not have a canonical labeling.
\end{obs}

\begin{lem} [\cite{KM21}] \label{lem: 4.2 Gunjan}
Suppose that $G$ is an even cycle and $\HH = (L,H)$ is a cover of $G$ where $|L(v)| \geq 2$ for each $v \in V(G)$. Then $G$ does not admit an $\mathcal{H}$-coloring if and only if $\HH$ is a $2$-fold cover with a twisted canonical labeling.
\end{lem}

\begin{cor} \label{cor: GunjanTwist}
Suppose $G = C_{2k+2}$ and $\HH = (L,H)$ is a $2$-fold cover of $G$ for which there is no $\HH$-coloring. Then the spanning subgraph of $H$ containing only the cross edges of $H$ is a copy of $C_{4k+4}$.
\end{cor}

\begin{proof}
Suppose that the vertices of $G$ in cyclic order are $v_1,\ldots,v_{2k+2}$. Since $G$ does not admit an $\HH$-coloring, by Lemma~\ref{lem: 4.2 Gunjan}, $\HH$ is a $2$-fold cover with a twisted canonical labeling. Suppose $G' = G - \{v_1v_{2k+2}\}$ and $H' = H - E_H(L(v_1),L(v_{2k+2}))$. Let $\HH' = (L,H')$, and note $\HH'$ is a $2$-fold cover of $G'$. By Proposition~\ref{pro: tree}, $\HH'$ has a canonical labeling. Suppose the vertices of $H'$ are renamed according to this labeling. Since $\HH$ has a twisted canonical labeling, we know by Observation~\ref{obsrv: labeling}, $\HH$ does not have a canonical labeling. Thus, $(v_1,1)(v_{2k+2},2),(v_1,2)(v_{2k+2},1) \in E(H)$. Notice $E = \{(v_i,1)(v_i,2) : i \in [2k+2]\}$ are the non-cross edges of $H$.  Finally, $H - E = C_{4k+4}$.
\end{proof}

The proof of the next result will show that the converse of Corollary~\ref{thm: general} is not true, and it will give a lower bound on the DP color function threshold of cones of disjoint unions of cycles.

\begin{pro}\label{pro: (M,3)}
Suppose $n \geq 2$. Let $M = K_1 \vee G$, where $G$ is the disjoint union of cycles $C_{k_i}$ for $i\in[n]$, with each $k_i \ge 3$.
If $k_i$ is odd for any $i \in [n]$ or $n \geq 3$, then $P_{DP}(M,3) = 0.$ If $n=2$ and $k_i$ is even for each $i \in [n]$, then $$P_{DP}(M,3) = 3 = \frac{P_{DP}(M_1,3)P_{DP}(M_2,3)}{3} < 12 = P(M,3).$$ Consequently, $\tau_{DP}(M) \geq 4$.
\end{pro}

\begin{proof}
Suppose that the vertices of the copy of $C_{k_i}$ in $M_i$ in cyclic order are $v_{i,1},\ldots,v_{i,k_i}$ for each $i \in [n]$.  We will prove the following statements: (1) $P_{DP}(M,3) = 0$ when $k_i$ is odd for some $i \in [n]$, (2) $P_{DP}(M,3) = 3$ when $n = 2$ and $k_1$ and $k_2$ are even, and (3) $P_{DP}(M,3) = 0$ when $n \geq 3$ and $k_i$ is even for each $i \in [n]$. For (1), note that $\chi(M)=4$ which means $P(M,3) = 0$ which implies that $P_{DP}(M,3) = 0$.

With our focus on (2), we suppose that $n = 2$ and $k_1$ and $k_2$ are even.  Suppose $\HH = (L,H)$ is an arbitrary full $3$-fold cover of $M$. For each $i \in [2]$, assume $\HH_i = (L_i,H_i)$ is the separated cover of $M_i$ obtained from $\HH$ where $u_{i,1} = w_i$.  Let $f: L_1(w_1) \rightarrow L_2(w_2)$ be the function given by $f((w_1,j)) = (w_2,j)$ for each $j \in [3]$. Note $\HH = (L,H)$ is the $f$-amalgamated $3$-fold cover of $M$ obtained from $\HH_1$ and $\HH_2$. Since $M_i - E(M_i[\{v_{i,j} : j \in [k_i]\}])$ is a spanning tree of $M_i$ for each $i \in [2]$, we can rename the vertices in $L(v_{i,j})$, while maintaining the names of the vertices in $L(w_i)$, so that $(w_i,l)(v_{i,j},l) \in E(H_i)$ for each $j \in [k_i]$ and each $l \in [3]$ (see the proof of Proposition 21 in~\cite{KM19}).  

We will show $P_{DP}(M, \HH) \geq 3$ which will imply $P_{DP}(M,3) \geq 3$. In the case $N((w_i,j),\HH_i) \geq 1$ for each $i \in [2]$ and $j \in [3]$, by Lemma~\ref{lem: upperGen}, $P_{DP}(M,\HH) = \sum_{j=1}^{3} N((w_1,j),\HH_1)N((w_2,j),\HH_2) \geq 3$.  So, assume without loss of generality that $N((w_1,1),\HH_1) = 0$. Let $\HH_1^{(1)}$ be the cone reduction of $\HH_1$ by $(w_1,1)$, and let $H' = H_1[\bigcup_{s=1}^{k_1} \{(v_{1,s},2),(v_{1,s},3)\}] - \{(v_{1,s},2)(v_{1,s},3): s \in [k_1]\}$. Notice $\HH_1^{(1)}$ is a $2$-fold cover of an even cycle, $H'$ is a spanning subgraph of $H_1^{(1)}$, and there are no $\HH_1^{(1)}$-colorings of $H'$. Then by Corollary~\ref{cor: GunjanTwist}, $H' = C_{2k_1}$.  This implies $H_1[\{(v_{1,s},1) : s \in [k_1]\}] = C_{k_1}$ and $H_1[\{(v_{1,s},2) : s \in [k_1]\}]$ and $H_1[\{(v_{1,s},3) : s \in [k_1]\}]$ are proper subgraphs of $C_{k_1}$.  Also notice that $\{(w_1,2)\} \cup \{(v_{1,s},2+(-1)^s) : s \in [k_1]\}$, $\{(w_1,2)\} \cup \{(v_{1,s},2-(-1)^s) : s \in [k_1]\}$, $\{(w_1,3)\} \cup \{(v_{1,s},(3+(-1)^s)/2) : s \in [k_1]\}$, and $\{(w_1,3)\} \cup \{(v_{1,s},(3-(-1)^s)/2) : s \in [k_1]\}$ are $\HH_1$-colorings of $M_1$. Thus, $N((w_1,2),\HH_1) \geq 2$ and $N((w_1,3),\HH_1) \geq 2$. If $N((w_2,2),\HH_2) \geq 1$ and $N((w_2,3),\HH_2) \geq 1$, then by Lemma~\ref{lem: upperGen}, $P_{DP}(M,\HH) \geq 4$. So, assume without loss of generality that $N((w_2,2),\HH_2) = 0$. By a similar argument, $N((w_2,3),\HH_2) \geq 2$. Then by Lemma~\ref{lem: upperGen}, $P_{DP}(M,\HH) \geq 4$. Thus, $P_{DP}(M,3) \geq 3$.

To complete the proof of (2), we will now construct a $3$-fold cover $\HH = (L,H)$ of $M$ with exactly three colorings. First, for each $i \in [2]$, we construct a $3$-fold cover $\HH_i = (L_i,H_i)$ of $M_i$. Let $L_1(x) = \{(x,j) : j \in [m]\}$ for each $x \in V(M_1)$. Construct edges in $H_1$ so that $L_1(x)$ is a clique in $H_1$ for each $x \in V(M_1)$. Whenever $uv \in E(M_1) - \{v_{1,1}v_{1,k_1}\}$, create an edge between $(u,j)$ and $(v,j)$ for each $j \in [3]$.  Finally, construct the edges $(v_{1,1},1)(v_{1,k_1},2),(v_{1,1},2)(v_{1,k_1},3),(v_{1,1},3)(v_{1,k_1},1)$. Construct $\HH_2$ similarly. By the proof of Theorem~\ref{thm: wheel}, $N((w_i,j),\HH_i) = 1$ for each $i \in [2]$ and $j \in [3]$. Let $f : L_1(w_1) \rightarrow L_2(w_2)$ be the function defined by $f((w_1,j)) = (w_2,j)$. Let $\HH$ be the $f$-amalgamated $m$-fold cover of $G$ obtained from $\HH_1$ and $\HH_2$. By Lemma~\ref{lem: upperGen}, $P_{DP}(M, \mathcal{H}) = 3$.  Thus, $P_{DP}(M,3) = 3 = P_{DP}(M_1,3)P_{DP}(M_2,3)/3$. 

For (3), for each $i \in [n]$, we construct a $3$-fold cover $\HH_i = (L_i,H_i)$ of $M_i$. For each $i \in [n]$ let $L_i(x) = \{(x,j) : j \in [m]\}$ for each $x \in V(M_i)$. Construct edges in $H_1$ so that $L_1(x)$ is a clique in $H_1$ for each $x \in V(M_1)$. Whenever $uv \in E(M_1) - \{v_{1,1}v_{1,k_1}\}$, create an edge between $(u,j)$ and $(v,j)$ for each $j \in [3]$.  Finally, construct the edges $(v_{1,1},1)(v_{1,k_1},1),(v_{1,1},2)(v_{1,k_1},3),(v_{1,1},3)(v_{1,k_1},2)$.  

Construct edges in $H_2$ so that $L_2(x)$ is a clique in $H_2$ for each $x \in V(M_2)$. Whenever $uv \in E(M_1) - \{v_{2,1}v_{2,k_2}\}$, create an edge between $(u,j)$ and $(v,j)$ for each $j \in [3]$.  Finally, construct the edges $(v_{2,1},1)(v_{2,k_2},3),(v_{2,1},2)(v_{2,k_2},2),(v_{2,1},3)(v_{2,k_2},1)$.  

Construct edges in $H_3$ so that $L_3(x)$ is a clique in $H_3$ for each $x \in V(M_3)$. Whenever $uv \in E(M_3) - \{v_{3,1}v_{3,k_3}\}$, create an edge between $(u,j)$ and $(v,j)$ for each $j \in [3]$.  Finally, construct the edges $(v_{3,1},1)(v_{3,k_3},2),(v_{3,1},2)(v_{3,k_3},1),(v_{3,1},3)(v_{3,k_3},3)$.  

For each $4 \leq i \leq n$ arbitrarily construct the edges of $H_i$ so that $\mathcal{H}_i$ is a cover of $M_i$.  By construction, $N((w_1,1),\HH_1) = 0$, $N((w_2,2),\HH_2) = 0$, and $N((w_3,3),\HH_3) = 0$.  For each $2 \leq i \leq n$, let $f_i: L_1(w_1) \rightarrow L_i(w_i)$ be the function given by $f_i((w_1,j)) = (w_i,j)$ for each $j \in [3]$. Let $F = (f_2,\ldots,f_l)$. Then let $\HH = (L,H)$ be the $F$-amalgamated $3$-fold cover of $M$ obtained from $\HH_1,\ldots,\HH_n$. By Lemma~\ref{lem: upperGen}, $P_{DP}(M, \mathcal{H}) = 0$ which implies that $P_{DP}(M,3) = 0$.
\end{proof}

The next theorem will provide an upper bound on the DP color function threshold of cones of the disjoint unions of cycles. First, we prove a technical lemma.
\begin{lem}\label{lem: cyclestechnical}
For any $m \geq 5$ and non-negative integer $s$ satisfying $ s \leq m - 2$, $$\left(1 - \frac{m}{(m-2)^4}\right)^{s}\left(1 + \frac{1}{(m-1)(m-2)} - \frac{m}{(m-2)^4}\right)^{m-s} > 1.$$
\end{lem}
\begin{proof}
Let $A_s = \left(1 - (m/(m-2)^4)\right)^{s}\left(1 + (1/((m-1)(m-2))) - (m/(m-2)^4)\right)^{m-s}$. Notice that for each $m \geq 5$ and $0 \leq s \leq m-1$, $A_{s+1}/A_s = (1 - (m/(m-2)^4))(1 + (1/((m-1)(m-2))) - (m/(m-2)^4))^{-1} < 1$. Thus, $A_{m-2} \leq A_s$ for each $0 \leq s \leq m - 2$. Let $f(m) = (1 - (m/(m-2)^4))^{m-2}(1 + (1/((m-1)(m-2))) - (m/(m-2)^4))^{2}$. To complete the proof we will show that
\begin{align*}
    \ln{(f(m))} = (m-2)\ln{\left(1 - \frac{m}{(m-2)^4}\right)} + 2\ln{\left(1 + \frac{1}{(m-1)(m-2)} - \frac{m}{(m-2)^4}\right)} > 0
\end{align*}
for each $m \geq 23$. Notice $\lim_{m \to \infty} \ln{(f(m))} = 0$. Furthermore,
\begin{align*}
    [\ln{(f(m))}]' &= \frac{3m+2}{m^4 - 8m^3 + 24m^2 - 33m + 16} \\ &- \frac{2(2m^4 - 18m^3 + 46m^2 - 51m + 22)}{(m-1)(m-2)(m^5 - 9m^4 + 33m^3 - 63m^2 + 61m - 24)} + \ln{\left(1 - \frac{m}{(m-2)^4}\right)}.
\end{align*}
Using a computer algebra system, it is easy to see that the sum of the rational terms of $[\ln{(f(m))}]'$ are negative for each $m \geq 23$. Thus, $[\ln{(f(m))}]' < 0$ for each $m \geq 23$ which implies that $\ln{(f(m))} > 0$ for each $m \geq 23$. It is easy to directly verify that $f(m) > 1$ for each $5 \leq m < 23$.
\end{proof}

\begin{thm}\label{thm: manycycles}
Suppose $m,n \in \N$, $m \geq 5$, and $n \geq 2$. Let $M = K_1 \vee G$, where $G$ is the disjoint union of cycles $C_{k_i}$ for $i\in[n]$, with each $k_i \ge 3$.
Then $P_{DP}(M,m) = P(M,m).$  Consequently, $\tau_{DP}(M) \leq 5$.
\end{thm}
\begin{proof}
Let $\HH = (L,H)$ be an arbitrary $m$-fold cover of $M$. For each $i \in [n]$, let $\HH_i = (L_i,H_i)$ be the separated cover of $M_i$ obtained from $\HH$ where $u_{i,1} = w_i$. For each $2 \leq i \leq l$, let $f_i: L_1(w_1) \rightarrow L_i(w_i)$ be the function given by $f_i((w_1,j)) = (w_i,j)$ for each $j \in [m]$. Let $F = (f_2,\ldots,f_n)$. Then $\HH = (L,H)$ is the $F$-amalgamated $m$-fold cover of $M$ obtained from $\HH_1,\ldots,\HH_n$. Let $s_i$ be the number of level vertices of $\HH_i$ in $L_i(w_i)$ for each $i \in [n]$. Note $s_i \in \{0\} \cup ([m] - \{m-1\})$. If $L_i(w_i)$ has a level vertex of $\HH_i$, let $a_i$ be a level vertex of $\HH_i$ in the fewest number of $\HH_i$-colorings for each $i \in [n]$. If $L_i(w_i)$ has a vertex that is not a level vertex of $\HH_i$, let $b_i$ be such a vertex in the fewest number of $\HH_i$-colorings for each $i \in [n]$. If $k_i$ is odd, since $P_{DP}(C_{k_i},m-1) = P(C_{k_i},m-1)$, then $N(a_i,\HH_i)$ and $N(b_i,\HH_i)$ are both bounded below by $P(C_{k_i},m-1)$. If $k_i$ is even, then by Lemma~\ref{lem: sethFormula}, $N(a_i,\HH_i) \geq (m-2)^{k_i} + s_i - 2$ and $N(b_i,\HH_i) \geq ((m^2 - 3m + 3)(m-2)^{k_i-1} - m + 2 + s_i(m-1))/(m-1)$ for each $i \in [l]$.

By Lemma~\ref{lem: upperGen}, $P_{DP}(M,\HH) = \sum_{j=1}^{m} \prod_{i=1}^{n} N((w_i,j),\HH_i)$. By the AM-GM Inequality, $P_{DP}(M,\HH) \geq m\left(\prod_{j=1}^{m} \prod_{i=1}^{n} N((w_i,j),\HH_i)\right)^{1/m} \geq m\left(\prod_{i=1}^{n} N(a_i,\HH_i)^{s_i}N(b_i,\HH_i)^{m-s_i}\right)^{1/m}$. Let $Y_{k_i} = P(M_i,m)/m$. Notice that when $k_i$ is odd, $Y_{k_i} = (m-2)^{k_i} - (m-2)$, and when $k_i$ is even, $Y_{k_i} = (m-2)^{k_i} + (m-2)$. We will show that $N(a_i,\HH_i)^{s_i}N(b_i,\HH_i)^{m-s_i} \geq Y_{k_i}^{m}$ for each $i \in [n]$, $0 \leq s \leq m - 2$, and $m \geq 5$. This will complete the proof since it will imply that $P_{DP}(M,m) \geq m\left(\prod_{i=1}^{n} N(a_i,\HH_i)^{s_i}N(b_i,\HH_i)^{m-s_i}\right)^{1/m} \geq mY_{k_i}^n = P(M,m)$ for each $m \geq 5$.

Notice that when $k_i$ is odd, $N(a_i,\HH_i)^{s_i}N(b_i,\HH_i)^{m-s_i} \geq P(C_{k_i},m-1)^m = Y_{k_i}^m$. So assume $k_i$ is even. We calculate,
\begin{align*}
    N(a_i,\HH_i)^{s_i}N(b_i,\HH_i)^{m-s_i} &\geq (Y_{k_i} + s_i - m)^{s_i}\left(\left(\frac{m^2 - 3m + 3}{m^2 - 3m + 2}\right)Y_{k_i} + s_i - m + 1\right)^{m-s_i}\\
    &= Y_{k_i}^{m}\left(1 + \frac{s_i - m}{Y_{k_i}}\right)^{s_i}\left(1 + \frac{1}{(m-1)(m-2)} + \frac{s_i - m + 1}{Y_{k_i}}\right)^{m-s_i}.
\end{align*}
Now we will show that $$Z_{k_i} = \left(1 + \frac{s_i - m}{Y_{k_i}}\right)^{s_i}\left(1 + \frac{1}{(m-1)(m-2)} + \frac{s_i - m + 1}{Y_{k_i}}\right)^{m-s_i} \geq 1.$$ Clearly, $Z_{k_i} = 1$ when $s_i = m$. Notice $s_i - m + 1 < 0$ for all $0 \leq s_i \leq m-2$. Thus, $Z_{k_i} \geq Z_4$. By Lemma~\ref{lem: cyclestechnical}, for each $0 \leq s_i \leq m-2$,
\begin{align*}
    Z_4 &> \left(1 - \frac{m}{(m-2)^4}\right)^{s_i}\left(1 + \frac{1}{(m-1)(m-2)} - \frac{m}{(m-2)^4}\right)^{m-s_i} \geq 1.
\end{align*}
\end{proof}

Next, Proposition~\ref{pro: m=4} shows the conditions that allow for the DP color function threshold to be $4$.

\begin{pro}\label{pro: m=4}
Suppose $n \geq 2$. Let $M = K_1 \vee G$, where $G$ is the disjoint union of cycles $C_{k_i}$ for $i\in[n]$, with each $k_i \ge 3$, and there exists at most one element $p \in [n]$ such that $k_p = 4$.
Then, $$P_{DP}(M,4) = P(M,4).$$ Consequently, $\tau_{DP}(M) = 4$.
\end{pro}
\begin{proof}
Let $\HH = (L,H)$ be an arbitrary $4$-fold cover of $M$. For each $i \in [n]$, let $\HH_i = (L_i,H_i)$ be the separated cover of $M_i$ obtained from $\HH$ where $u_{i,1} = w_i$. For each $2 \leq i \leq n$, let $f_i: L_1(w_1) \rightarrow L_i(w_i)$ be the function given by $f_i((w_1,j)) = (w_i,j)$ for each $j \in [4]$. Let $F=(f_2,\ldots,f_n)$. Then $\HH = (L,H)$ is the $F$-amalgamated $4$-fold cover of $M$ obtained from $\HH_1,\ldots,\HH_n$. Let $s_i$ be the number of level vertices in $L_i(w_i)$ for each $i \in [l]$. Note $s_i \in \{0,1,2,4\}$. If $L_i(w_i)$ has a level vertex, let $a_i$ be a level vertex of $\HH_i$ in the fewest number of $\HH_i$-colorings for each $i \in [n]$. If $L_i(w_i)$ has a vertex that is not a level vertex of $\HH_i$, let $b_i$ be such a vertex in the fewest number of $\HH_i$-colorings for each $i \in [n]$. If $k_i$ is odd, since $P_{DP}(C_{k_i},3) = P(C_{k_i},3)$, $N(a_i,\HH_i)$ and $N(b_i,\HH_i)$ are both bounded below by $P(C_{k_i},3) = P(M_i,4)/4$. If $k_i$ is even, then by Lemma~\ref{lem: sethFormula}, $N(a_i,\HH_i) \geq 2^{k_i} + s_i - 2$ and $N(b_i,\HH_i) \geq (7(2)^{k_i-1} + 3s_i - 2)/3$ for each $i \in [n]$.

By Lemma~\ref{lem: upperGen}, $P_{DP}(M,\HH) = \sum_{j=1}^{4} \prod_{i=1}^{n} N((w_i,j),\HH_i)$. To complete the proof, we will show that $P_{DP}(M,\HH) \geq P(M,4)$ in each of the following two cases: (1) $k_i \not= 4$ for each $i \in [n]$ and (2) $k_p = 4$ for some $p \in [n]$. For case (1), by the AM-GM Inequality, $P_{DP}(M,\HH) \geq 4\left(\prod_{j=1}^{4} \prod_{i=1}^{n} N((w_i,j),\HH_i)\right)^{1/4} \geq 4\left(\prod_{i=1}^{n} (N(a_i,\HH_i)^{s_i} N(b_i,\HH_i)^{4-s_i})\right)^{1/4}$. We will show that $N(a_i,\HH_i)^{s_i}N(b_i,\HH_i)^{4-s_i} \geq (P(M_i,4)/4)^{4}$ for each $i \in [n]$. Whenever $k_i$ is odd, $N(a_i,\HH_i)^{s_i}N(b_i,\HH_i)^{4-s_i} \geq P(C_{k_i},3)^4 = (P(M_i,4)/4)^4$. When $k_i$ is even, we know $k_i \geq 6$, and we calculate,
\begin{align*}
    N(a_i,\HH_i)^{s_i}N(b_i,\HH_i)^{4-s_i} &\geq (2^{k_i} + s_i - 2)^{s_i}\left(\frac{7(2)^{k_i-1} + 3s_i - 2}{3}\right)^{4-s_i}\\
    &\geq (2^{k_i} + s_i - 2)^{s_i}(2^{k_i} + s_i + 10)^{4-s_i}\\
    &\geq (2^{k_i} + 2)^4 = \left(\frac{P(M_i,4)}{4}\right)^{4}.
\end{align*}
Thus,
\begin{align*}
    P_{DP}(M,\HH) &\geq 4\left(\prod_{i=1}^{n} N(a_i,\HH_i)^{s_i}N(b_i,\HH_i)^{4-s_i}\right)^{1/4}\\
    &\geq \frac{\prod_{i=1}^{n} P(M_i,4)}{4^{n-1}} = P(M,4).
\end{align*}

In case (2), without loss of generality, assume that there are non-negative integers $a$ and $b$ such that $k_1 = 4$, $k_i > 4$ is even and $s_i \in \{0,1,2\}$ for each $2 \leq i \leq a + 1$, $k_i > 4$ is even and $s_i = 4$ for each $a + 2 \leq i \leq a + b + 1$, and $k_i$ is odd for each $a + b + 2 \leq i \leq n$. Note it is possible for $a = 0$, $b = 0$, or $a + b = n - 1$. Also note that when $s_i = 4$ or $k_i$ is odd, $N((w_i,j),\HH_i) \geq P(M_i,4)/4$ for each $j \in [4]$. Let $K_{s_1} = \prod_{j=1}^{4} N((w_1,j),\HH_1) = (14 + s_1)^{s_1}(18 + s_1)^{4-s_1}$. We observe that $18^4/K_{s_1} \geq 1$, and we see that
\begin{align*}
    P_{DP}(M,\HH) = \sum_{j=1}^{4} \prod_{i=1}^{n} N((w_i,j),\HH_i) \geq \left(\prod_{i=a+2}^{n} \frac{P(M_i,4)}{4}\right) \left(\sum_{j=1}^{4} \prod_{i=1}^{a+1} N((w_i,j),\HH_i)\right).
\end{align*}
Note that $\sum_{j=1}^{4} N((w_1,j),\HH_1) = P_{DP}(M_1,\HH_1) \geq P_{DP}(M_1,4) = 72$ (the last equality follows from Theorem~\ref{thm: wheel}). When $a = 0$,
\begin{align*}
    P_{DP}(M,\HH) \geq \left(\prod_{i=2}^{n} \frac{P(M_i,4)}{4}\right) \left(\sum_{j=1}^{4} N((w_1,j),\HH_1)\right) \geq \frac{72\prod_{i=2}^{n} P(M_i,4)}{4^{n-1}} = P(M,4).
\end{align*}
So, assume that $a \geq 1$. We notice
\begin{align*}
    P_{DP}(M,\HH) &\geq \left(\prod_{i=a+2}^{n} \frac{P(M_i,4)}{4}\right) \left(\sum_{j=1}^{4} \prod_{i=1}^{a+1} N((w_i,j),\HH_i)\right)\\
    &\geq 4\left(\prod_{i=a+2}^{n} \frac{P(M_i,4)}{4}\right) \left(\prod_{j=1}^{4} \prod_{i=1}^{a+1} N((w_i,j),\HH_i)^{1/4}\right) \textrm{ (by the AM-GM Inequality)}\\
    &\geq 4K_{s_1}^{1/4}\left(\prod_{i=a+2}^{n} \frac{P(M_i,4)}{4}\right) \left(\prod_{i=2}^{a+1} (N(a_i,\HH_i)^{s_i} N(b_i,\HH_i)^{4-s_i})^{1/4}\right).
\end{align*}
We will show that $N(a_i,\HH_i)^{s_i}N(b_i,\HH_i)^{4-s_i} \geq \left(18(P(M_i,4)/4)\right)^4/K_{s_1}$ for each $2 \leq i \leq a + 1$. Recall that $N(a_i,\HH_i) \geq 2^{k_i} + s_i - 2$, $N(b_i,\HH_i) \geq (7(2)^{k_i-1} + 3s_i - 2)/3$, and $k_i \geq 6$ when $2 \leq i \leq a+1$. We calculate,
\begin{align*}
    N(a_i,\HH_i)^{s_i}N(b_i,\HH_i)^{4-s_i} &\geq (2^{k_i} + s_i - 2)^{s_i}\left(\frac{7(2)^{k_i-1} + 3s_i - 2}{3}\right)^{4-s_i}\\
    &\geq (2^{k_i} + s_i - 2)^{s_i}(2^{k_i} + s_i + 10)^{4-s_i}\\
    &\geq \frac{18^4}{K_{s_1}}(2^{k_i} + 2)^4 = \frac{18^4}{K_{s_1}}\left(\frac{P(M_i,4)}{4}\right)^4.
\end{align*}
Thus,
\begin{align*}
    P_{DP}(M,\HH) &\geq 4K_{s_1}^{1/4}\left(\prod_{i=a+2}^{n} \frac{P(M_i,4)}{4}\right) \left(\prod_{i=2}^{a+1} (N(a_i,\HH_i)^{s_i} N(b_i,\HH_i)^{4-s_i})^{1/4}\right)\\
    &\geq 4K_{s_1}^{1/4}\left(\frac{\prod_{i=a+2}^{n} P(M_i,4)}{4^{n-a-1}}\right) \left(\prod_{i=2}^{a+1} \frac{18P(M_i,4)}{4K_{s_1}^{1/4}}\right)\\
    &= \left(\frac{18}{K_{s_1}^{1/4}}\right)^{a-1}\frac{72\prod_{i=2}^{n} P(M_i,4)}{4^{n-1}} \geq \frac{72\prod_{i=2}^{n} P(M_i,4)}{4^{n-1}} = P(M,4).
\end{align*}
\end{proof}

Proposition~\ref{pro: (C_4,4)} below is the last result needed to prove Theorem~\ref{cor: coneofcycles}.  The following observation, which allows us to think of $\HH$-colorings as vertex labelings, is a useful perspective to have in mind for the proof of Proposition~\ref{pro: (C_4,4)}.

\begin{obs}\label{obsrv: backwardPerspective}
Suppose $G$ is an arbitrary graph and $\HH = (L,H)$ is an arbitrary $m$-fold cover of $G$. Let $f: V(G) \rightarrow [m]$ be a function that satisfies: if $uv \in E(G)$, then $(u,f(u))(v,f(v)) \notin E(H)$. Let $\FF$ be the set of all such $f$, $\II$ be the set of all $\HH$-colorings of $G$, and $B: \FF \rightarrow \II$ be the function defined by $B(f) = \{(u,f(u)) : u \in V(G)\}$. Then $B$ is a bijection.
\end{obs}

\begin{pro}\label{pro: (C_4,4)}
Let $n \geq 2$. Let $M = K_1 \vee G$, where $G$ is the disjoint union of cycles $C_{k_i}$ for $i\in[n]$, with each $k_i \ge 3$, and there exist distinct $a,b \in [n]$ such that $k_a = k_b = 4$.
Then $$P_{DP}(M,4) < P(M,4).$$ Consequently, $\tau_{DP}(M) \geq 5$.
\end{pro}

\begin{proof}
Without loss of generality, suppose $k_1 = k_2 = 4$. For each $i \in [n]$, suppose that the vertices of the copy of $C_{k_i}$ in $M_i$ in cyclic order are $v_{i,1},\ldots,v_{i,k_i}$. For each $i \in [n]$, we will construct an $m$-fold cover $\HH_i = (L_i,H_i)$ of $M_i$. For each $i \in [n]$, let $L_i(x) = \{(x,j) : j \in [4]\}$ for each $x \in V(M_i)$, and let $V(H_i) = \bigcup_{x \in V(M_i)} L_i(x)$. For each $i \in [n]$, construct edges so that $H_i[L_i(x)]$ is a complete graph for each $x \in V(M_i)$, so that $(w_i,j)(v_{i,k},j) \in E(H_i)$ for each $k \in [k_i]$ and each $j \in [4]$, and so that $H_i$ contains the edges in $\{(v_{i,k},j)(v_{i,k+1},j) : k \in [k_i - 1], j \in [4]\}$.
Construct edges so that for each $j \in [2]$, $(v_{1,1},j)(v_{1,4},j) \in E(H_1)$ and $(v_{2,1},j+2)(v_{2,4},j+2) \in E(H_2)$. Also construct edges so that $(v_{1,1},3)(v_{1,4},4),(v_{1,1},4)(v_{1,4},3) \in E(H_1)$ and $(v_{2,1},1)(v_{2,4},2),(v_{2,1},2)(v_{2,4},1) \in E(H_2)$. For each $i \geq 3$, construct edges so that $(v_{i,1},j)(v_{i,k_i},j) \in E(H_i)$ for each $j \in [4]$.

Let $\II_j$ be the set of $\HH_1$-colorings that contain $(w_1,j)$ for each $j \in [4]$. Notice $N((w_1,j),\HH_1) = |\II_j|$ for each $j \in [4]$. Also notice that $|\II_1| = |\II_2|$ and $|\II_3| = |\II_4|$. Let $\HH_1^{(j)}$ be the cone reduction of $\HH_1$ by $(w_1,j)$. We can determine $\II_j$ for some $j \in [4]$ by determining the number of independent sets of size $4$ of $H_1^{(j)}$. Notice this can be done by counting the number of proper $3$-colorings of a $P_4$ and subtracting the number of those colorings that color the end vertices $c_1,c_2 \in [j]$ such that $(v_{1,1},c_1)(v_{1,4},c_2) \in E(H_1^{(j)})$. When $c_1 = c_2$, the number of colorings that must be subtracted is equal to the number of proper $3$-colorings of a $C_3$ that color one of the vertices $c_1 \in [j]$ which is $P(C_3,3)/3$. When $c_1 \not= c_2$, the number of colorings that must be subtracted is equal to the number of proper $3$-colorings of a $C_4$ that color one of the vertices $c_1 \in [4]$ and an adjacent vertex $c_2 \in [j] - \{c_1\}$ which is $P(C_4,3)/6$.
By construction, $$|\II_1| = |\II_2| = P(P_4,3) - \frac{P(C_3,3)}{3} - \frac{2P(C_4,3)}{6} = 3(2)^3 - \frac{\left(2^3-2\right)}{3} - \frac{2\left(2^3+1\right)}{3} = 16.$$
By construction, $$|\II_3| = |\II_4| = P(P_4,3) - \frac{2P(C_3,3)}{3} = 3(2)^3 - \frac{2(2^3-2)}{3} = 20.$$
Thus, $N((w_1,1),\HH_1) = N((w_1,2),\HH_1) = 16$ and $N((w_1,3),\HH_1) = N((w_1,4),\HH_1) = 20$, and by a similar argument, $N((w_2,1),\HH_2) = N((w_2,2),\HH_2) = 20$ and $N((w_2,3),\HH_2) = N((w_2,4),\HH_2) = 16$. Since each $\HH_i^{(j)}$ has a canonical labeling when $3 \leq i \leq n$, $N((w_i,j),\HH_i) = P(C_{k_i},3) = P(M_i,4)/4$ for each $3 \leq i \leq n$ and $j \in [4]$. 

For each $i \in [n-1]$, let $f_{i+1}: L_1(w_1) \rightarrow L_{i+1}(w_{i+1})$ be the function given by $f_{i+1}((w_1,j)) = (w_{i+1},j)$ for each $j \in [4]$.  Let $\HH = (L,H)$ be the $F$-amalgamated $m$-fold cover of $M$ obtained from $\HH_1,\ldots,\HH_n$ where $F = (f_2,\ldots,f_n)$. Notice $N((w_1,j),\HH_1)N((w_2,j),\HH_2) = 320$ for each $j \in [4]$. Also notice $N((w_i,j_1),\HH_i) = N((w_i,j_2),\HH_i)$ for each $j_1,j_2 \in [4]$ and $3 \leq i \leq n$. Thus, by Lemma~\ref{lem: upperGen},
\begin{align*}
    P_{DP}(M,\HH) = \sum_{j=1}^{4} \prod_{i=1}^{n} N((w_i,j),\HH_i) &= 320\sum_{j=1}^{4} \prod_{i=3}^{n} N((w_i,j),\HH_i)\\
    &= 1280\prod_{i=3}^{n} \frac{P_{DP}(M_i,4)}{4}\\
    &< 1296\prod_{i=3}^{n} \frac{P(M_i,4)}{4} = P(M,4).
\end{align*}
\end{proof}

Putting all these results together gives us Theorem~\ref{cor: coneofcycles}.

\begin{proof}[Proof of Theorem~\ref{cor: coneofcycles}]
By Proposition~\ref{pro: (M,3)}, $\tau_{DP}(M) \geq 4$, and by Theorem~\ref{thm: manycycles}, $\tau_{DP}(M) \leq 5$. Assume that there is at most one element $p \in [n]$ such that $k_p = 4$. Then by Proposition~\ref{pro: (C_4,4)}, $\tau_{DP}(M) \leq 4$. Now assume that there are at least two elements $a,b \in [n]$ such that $k_a = k_b = 4$. Then by Proposition~\ref{pro: m=4}, $\tau_{DP}(M) \geq 5$.
\end{proof}

{\bf Acknowledgment.}  This paper is a combination of research projects conducted with undergraduate students: Jack Becker, Jade Hewitt, Michael Maxfield, David Spivey, Seth Thomason, and Tim Wagstrom at the College of Lake County during the summer and fall of 2020. The support of the College of Lake County is gratefully acknowledged.


\begin{thebibliography}{99}
{\small

\bibitem{A00} N. Alon, Degrees and choice numbers, \emph{Random Structures Algorithms} 16 (2000), 364-368.

\bibitem{AT92} N. Alon and M. Tarsi, Colorings and orientations of graphs, \emph{Combinatorica} 12 (1992), 125-134.

\bibitem{B16} A. Bernshteyn, The asymptotic behavior of the correspondence chromatic number, \emph{Discrete Mathematics}, 339 (2016), 2680-2692.

\bibitem{B17} A. Bernshteyn, The Johansson-Molloy Theorem for DP-coloring, \emph{Random Structures \& Algorithms} 54:4 (2019), 653-664.

\bibitem{BK17} A. Bernshteyn and A. Kostochka, On differences between DP-coloring and list coloring, \emph{Siberian Advances in Mathematics} 21:2 (2018),  61-71.

\bibitem{BK18} A. Bernshteyn, A. Kostochka, and X. Zhu, DP-colorings of graphs with high chromatic number, \emph{European J. of Comb.} 65 (2017), 122-129.

\bibitem{BF16} J. Beier, J. Fierson, R. Haas, H. M. Russel, K. Shavo, Classifying coloring graphs, \emph{Discrete Mathematics} 339 (2016),  no. 8, 2100-2112.

\bibitem{B94} N. Biggs, (1994) \emph{Algebraic graph theory}.  New York, NY: Cambridge University Press.

\bibitem{B12} G. D. Birkhoff, A determinant formula for the number of ways of coloring a map, \emph{The Annals of Mathematics} 14 (1912), 42-46.

\bibitem{D61} G. A. Dirac, On rigid circuit graphs, \emph{Abh. Math. Sem. Univ. Hamburg} 25 (1961), 71-76.

\bibitem{DKT05} F. Dong, K. M. Koh, K. L. Teo, \emph{Chromatic polynomials and chromaticity of graphs}, World Scientific, 2005.

\bibitem{D92} Q. Donner, On the number of list-colorings, \emph{J. Graph Theory} 16 (1992), 239-245.

\bibitem{DP15} Z. Dvo\v{r}\'{a}k and L. Postle, Correspondence coloring and its application to list-coloring planar graphs wihtout cycles of lengths 4 to 8, \emph{Journal of Combinatorial Theory Series B} 129 (2018), 38-54.

\bibitem{ET79} P. Erd\H{o}s, A. L. Rubin, and H. Taylor, Choosability in graphs, \emph{Congressus Numerantium} 26 (1979), 125-127.

\bibitem{FG65} D. R. Fulkerson and O. A. Gross, Incidence matrices and interval graphs, \emph{Pacific J. Math} 15 (1965), 835-855.

\bibitem{HK21} C. Halberg, H. Kaul, A. Liu, J. Mudrock, P. Shin, S. Thomason, On Polynomial Representations of the DP Color Function: Theta Graphs and Their Generalizations, \emph{submitted for publication}.

\bibitem{KM18} H. Kaul and J. Mudrock, Criticality, the list color function, and list coloring the Cartesian product of graphs, \emph{Journal of Combinatorics} 12 (2021), to appear.

\bibitem{KM19} H. Kaul and J. Mudrock, On the chromatic polynomial and counting DP-colorings of graphs, \emph{Advances in Applied Mathematics} 123 (2021), article 103121.

\bibitem{KM20} H. Kaul and J. Mudrock, Combinatorial Nullstellensatz and DP-coloring of Graphs, \emph{Discrete Mathematics} 343 (2020), article 112115.

\bibitem{KM21} H. Kaul, J. Mudrock, G. Sharma, Q. Stratton, DP-coloring the Cartesian Products of Graphs, \emph{in preparation}.

\bibitem{KO182} S-J. Kim and K. Ozeki, A sufficient condition for DP-4-colorability, \emph{Discrete Mathematics} 341 (2018), 1983-1986.

\bibitem{KN16} R. Kirov and R. Naimi, List coloring and $n$-monophilic graphs, \emph{Ars Combinatoria} 124 (2016), 329-340.

\bibitem{AS90} A. V. Kostochka and A. Sidorenko, Problem Session of the Prachatice Conference on Graph Theory, \emph{Fourth Czechoslovak Symposium on Combinatorics, Graphs and Complexity}, Ann. Discrete Math. 51 (1992), 380.

\bibitem{M17} M. Molloy, The list chromatic number of graphs with small clique number, \emph{Journal of Combinatorial Theory Series B} 134 (2019), 264-284.

\bibitem{M18} J. Mudrock, A note on the DP-chromatic number of complete bipartite graphs, \emph{Discrete Mathematics} 341 (2018), 3148-3151.

\bibitem{MT20} J. Mudrock and S. Thomason, Answers to two questions on the DP color function, \emph{Elect. Journal of Combinatorics}, accepted for publication.

\bibitem{T09} C. Thomassen, The chromatic polynomial and list colorings, \emph{Journal of Combinatorial Theory Series B} 99 (2009), 474-479.

\bibitem{V76} V. G. Vizing, Coloring the vertices of a graph in prescribed colors, \emph{Diskret. Analiz.} no. 29, \emph{Metody Diskret. Anal. v Teorii Kodovi Skhem} 101 (1976), 3-10.

\bibitem{WQ17} W. Wang, J. Qian, and Z. Yan, When does the list-coloring function of a graph equal its chromatic polynomial, \emph{Journal of Combinatorial Theory Series B} 122 (2017), 543-549.

\bibitem{W01} D. B. West, (2001) \emph{Introduction to Graph Theory}.  Upper Saddle River, NJ: Prentice Hall.

\bibitem{W32} H. Whitney, A logical expansion in mathematics, \emph{Bull. Amer. Math. Soc.} 38 (1932), 572-579.

}

\end{thebibliography}
\end{document}